\documentclass[11pt]{amsart}
\usepackage{amsmath,amssymb,amsthm,amscd}
\usepackage[a4paper,text={140mm,240mm},centering,headsep=5mm,footskip=10mm]{geometry}
\usepackage[matrix,arrow]{xy}
\numberwithin{equation}{section}

\theoremstyle{plain}
\newtheorem{theo}{Theorem}[section]
\newtheorem{lem}[theo]{Lemma}

\newtheorem{prop}[theo]{Proposition}
\newtheorem{conj}[theo]{Conjecture}
\newtheorem{coro}[theo]{Corollary}
\newtheorem{claim}[theo]{Claim}

\theoremstyle{definition}
\newtheorem{defi}[theo]{Definition}

\theoremstyle{remark}
\newtheorem{rem}[theo]{Remark}

\newtheorem{exam}[theo]{Example}

\newcommand{\lr}{\longrightarrow}
\newcommand{\Spec}{\mathrm{Spec}}
\newcommand{\AAA}{\mathbb{A}}

\newcommand{\Hom}{\text{\rm Hom}}

\newcommand{\cCS}{\mathcal C}
\newcommand{\Q}{{\mathbb Q}}
\newcommand{\Z}{{\mathbb Z}}

\newcommand{\OO}{{\mathcal{O}}}
\newcommand{\Gal}{\mathrm{Gal}}

\newcommand{\im}{\mathrm{im}}
\newcommand{\CH}{\mathrm{CH}}
\newcommand{\trdeg}{\mathrm{trdeg}}
\newcommand{\Br}{\mathrm{Br}}
\renewcommand{\CH}{\mathrm{CH}}
\newcommand{\et}{\text{\rm \'et}}
\newcommand{\hr}{\hookrightarrow}
\newcommand{\chara}{\text{\rm ch}}
\newcommand{\tr}{\text{\rm tr}}
\newcommand{\Zar}{\text{\rm Zar}}

\def\ord{{\mathrm{ord}}}      
\def\cd{{\mathrm{cd}}}      
\def\CH{{\mathrm{CH}}}      
\def\Coker{{\mathrm{Coker}}}
\def\Ker{{\mathrm{Ker}}}    
\def\Br{{\mathrm{Br}}}    
\def\cotor{{\mathrm{cotor}}}
\def\tors{{\mathrm{tors}}}  
\def\cotor{{\mathrm{cotor}}}  

\def\et{{\text{\'{e}t}}}    
\def\Gal{{\mathrm{Gal}}}    
\def\H{{\mathrm{H}}}        
\def\Hom{{\mathrm{Hom}}}    
\def\Image{{\mathrm{Image}}}   

\def\ker{{\mathrm{Ker}}}    
\def\ord{{\mathrm{ord}}}    

\def\ch{{\mathrm{ch}}}    
\def\trdeg{{\mathrm{trdeg}}}    

\def\et{{\text{\'{e}t}}}    

\def\cO{{\mathcal O}}

\def\cC{{\mathcal C}}

\def\cS{{\mathcal S}}

\def\cH{{\mathcal H}}
\def\cK{{\mathcal K}}

\def\LP{{\mathcal L}{\mathcal P}}

\def\Lam{\Lambda}

\def\k{\kappa}

\def\ol#1{\overline{#1}}

\def\us#1#2{\underset{#1}{#2}}

\def\sumd#1#2{\underset{x\in {#1}_{(#2)}}\bigoplus}
\def\sumcd#1#2{\underset{x\in {#1}^{(#2)}}\bigoplus}
\def\sumcdy#1#2{\underset{y\in {#1}^{(#2)}}\bigoplus}

\def\bC{{\mathbb C}}
\def\bR{{\mathbb R}}
\def\bZ{{\mathbb Z}}
\def\bQ{{\mathbb Q}}

\def\qz{{\bQ}/{\bZ}}
\def\qzl{{\bQ_\ell}/{\bZ_\ell}}

\def\zl{{\bZ}_\ell}

\def\ql{{\bQ}_\ell}

\def\zp{{\bZ}_p}
\def\qp{{\bQ}_p}

\def\nz{\bZ/n\bZ}

\def\lnz{\bZ/\ell^n\bZ}

\def\lmz{\bZ/\ell^m\bZ}

\def\indlim#1{\underset{{\underset{#1}{\longrightarrow}}}{\mathrm{lim}}\; }
\def\projlim#1{\underset{{\underset{#1}{\longleftarrow}}}{\mathrm{lim}}\; }

\def\rmapo#1{\overset{#1}{\longrightarrow}}

\def\lmapo#1{\overset{#1}{\longleftarrow}}
\def\isom{\overset{\simeq}{\longrightarrow}}

\def\tk{\tilde{k}}

\def\qfor{\quad\text{for }}
\def\qwith{\quad\text{with }}

\def\Xb{\ol X}

\def\Ln{\Lambda_n}
\def\Lm{\Lambda_m}
\def\Linfty{\Lambda_\infty}
\def\Mod{{Ab}}

\def\EX#1#2#3{E^{#3}_{#1,#2}(X)}

\def\Hempty#1#2{H_{#1}(#2)}

\def\HDL#1#2{H^D_{#1}(#2,\Lambda)}

\def\HLn#1#2{H_{#1}(#2,\Ln)}
\def\HLm#1#2{H_{#1}(#2,\Lm)}
\def\HLLn{H(-,\Ln)}
\def\HDLn#1#2{H^D_{#1}(#2,\Ln)}
\def\HDLinfty#1#2{H^D_{#1}(#2,\Linfty)}

\def\Hn#1#2{H_{#1}(#2,\Ln)}

\def\Hetnz#1#2{H^{\et}_{#1}(#2,\nz)}

\def\HetLam#1#2{H^{\et}_{#1}(#2,\Lambda)}

\def\Hinfty#1#2{H_{#1}(#2,\Linfty)}

\def\Het{H^{\et}}

\def\KC#1{KC_H(#1)}

\def\KCn#1{KC_H(#1,\Lambda_n)}
\def\KCinfty#1{KC_H(#1,\Lambda_\infty)}

\def\KH#1#2{KH_{#1}(#2)}

\def\KHn#1#2{KH_{#1}(#2,\Lambda_n)}
\def\KHnz#1#2{KH_{#1}(#2,\nz)}

\def\KHinfty#1#2{KH_{#1}(#2,\Lambda_\infty)}

\def\KHet#1#2{KH^{\et}_{#1}(#2)}
\def\KHetL#1#2{KH^{\et}_{#1}(#2,\Lam)}

\def\KHDL#1#2{KH^{D}_{#1}(#2,\Lam)}
\def\KHDLn#1#2{KH^{D}_{#1}(#2,\Ln)}

\def\KHcLn#1#2{KH_{#1}(#2/K,\Ln)}
\def\KHccLn#1#2{KH_{#1}(#2/U,\Ln)}
\def\KCcLn#1{KC(#1/K,\Ln)}
\def\KCccLn#1{KC(#1/U,\Ln)}
\def\KHccnz#1#2{KH_{#1}(#2/U,\nz)}
\def\KCccnz#1{KC(#1/U,\nz)}

\def\edgehom#1#2{\epsilon^{#1}_{#2}}
\def\pedgehom#1#2{\partial\epsilon^{#1}_{#2}}

\def\GH{G}

\def\graphHempty#1#2{\GH_{#1}(#2,\Lambda_H)}
\def\graphHempty#1#2{\GH_{#1}(#2)}
\def\graphHn#1#2{\GH_{#1}(#2,\Lambda_n)}
\def\graphHinfty#1#2{\GH_{#1}(#2,\Lambda_\infty)}

\def\graphhom#1#2{\gamma^{#1}_{#2}}

\def\graphedge#1#2{\gamma\epsilon^{#1}_{#2}}

\def\DWn#1#2{W_n\Omega^{#1}_{#2}}
\def\DWnlog#1#2{W_n\Omega^{#1}_{#2,log}}

\def\ol#1{\overline{#1}}

\def\HMXr#1{H^{#1}_M(X,\bZ(r))}

\def\nt{\ell^n}
\def\mt{\ell^m}

\def\hPhi{\hat{\Phi}}

\def\cCs{\cC_{/s}}

\def\Xd#1{X_{(#1)}}
\def\Yd#1{Y_{(#1)}}

\def\dimc{\dim_S}
\def\cCS{\cC_S}

\begin{document}

\title{
Cohomological Hasse principle and motivic cohomology for arithmetic schemes}
\author{Moritz Kerz}
\author{Shuji Saito}
\address{Moritz Kerz\\
Universit\"at Duisburg-Essen\\
Fachbereich Mathematik, Campus Essen\\
45117 Essen\\
Germany}
\email{moritz.kerz@uni-due.de}
\address{Shuji Saito\\
Interactive Research Center of Science, 
Graduate School of Science and Engineering,
 Tokyo Institute of Technology\\
Ookayama, Meguro\\
Tokyo 152-8551\\
Japan
}
\email{sshuji@msb.biglobe.ne.jp}
\date{ 2011-12-6}
\thanks{ The first author is supported by the DFG Emmy Noether-programme.}

\maketitle

\markboth{Moritz Kerz and Shuji Saito}{Cohomological Hasse principle and motivic cohomology}

\begin{abstract}
In 1985 Kazuya Kato formulated a fascinating framework of conjectures which
generalizes
the Hasse principle for the Brauer group of a global field to the so-called
cohomological Hasse principle for an arithmetic scheme $X$. In this paper we
prove 
the prime-to-characteristic part of the cohomological Hasse principle.
We also explain its implications on finiteness of motivic cohomology and 
special values of zeta functions.
\end{abstract}

\tableofcontents


\section*{Introduction} \label{intro}
\bigskip

\noindent
Let $K$ be a global field, 
namely a number field or a function field of dimension one over a finite field.
A fundamental fact in number theory due to Brauer-Hasse-Noether and Witt
is the Hasse principle for the Brauer group of $K$, which asserts
the injectivity of the restriction map:
\begin{equation}\label{intro.eq1}
\Br(K) \to \underset{v\in P_K}{\bigoplus} \Br(K_v),
\end{equation}
where $P_K$ is the set of places of $K$ and 
$K_v$ is the completion of $K$ at $v$?.
In 1985 Kato \cite{K} formulated a fascinating framework of conjectures
which generalizes this fact to higher dimensional \it arithmetic schemes $X$, \rm 
namely schemes of finite type over a finite field or the ring of 
integers in a number field or a local field. For an integer $n>0$, he defined a complex 
{ $KC^{(0)}(X,\nz)$} of Bloch-Ogus type (now called the Kato complex of $X$):
\begin{multline}\label{eq.KC1}
\cdots\rmapo{\partial} \sumd X a H^{a+1}(x,\nz(a))\rmapo{\partial} 
\sumd X {a-1} H^{a}(x,\nz(a-1))\rmapo{\partial}\cdots\\
\cdots\rmapo{\partial} \sumd X 1 H^{2}(x,\nz(1))\rmapo{\partial} \sumd X 0 H^{1}(x,\nz)
\end{multline}
Here $\Xd a$ denotes the set of points $x\in X$ such that $\dim\overline{\{x\}}=a$
with the closure $\overline{\{x\}}$ of $x$ in $X$, and the term in degree $a$ is 
the direct sum of the Galois cohomology $H^{a+1}(x,\nz(a))$ of the residue fields
$\k(x)$ for $x\in \Xd a$, where the coefficients $\nz(a)$ are the Tate twist 
(see \S\ref{homology} Lemma \ref{lemHK1smooth}).
In case $X$ is the {\it model} of a global field $K$ (namely $X=\Spec(\cO_K)$ for 
the integer ring $\cO_K$ of the number field $K$ or
a smooth projective curve over a finite field with the function field $K$), 
the Kato complex { $KC^{(0)}(X,\nz)$} is 
$$
H^2(K,\nz(1)) \rmapo{\partial} \sumd X 0 H^{1}(x,\nz)
$$
which is shown to be isomorphic to the $n$-torsion part of \eqref{intro.eq1} 
where the direct sum of $\Br(K_v)$ for the infinite places $v$ of $K$ is removed 
if $X=\Spec(\cO_K)$.
For an arithmetic scheme $X$ flat over $\Spec(\cO_K)$, 
we need modify \eqref{eq.KC1} to get the right Kato complex:
For a scheme $Y$ of finite type over $K$ or $K_v$ with $v\in P_K$, 
consider the following complex { $KC^{(1)}(Y,\nz)$}:
\begin{multline}\label{eq.KC2}
\cdots\rmapo{\partial} \sumd Y a H^{a+2}(x,\nz(a+1))\rmapo{\partial} 
\sumd Y {a-1} H^{a+1}(x,\nz(a))\rmapo{\partial}\cdots\\
\cdots\sumd Y 1 H^{3}(x,\nz(2))\rmapo{\partial} \sumd Y 0 H^{2}(x,\nz(1))\;,
\end{multline}
where the sum $\oplus_{x\in \Yd {a}}$ is put at degree $a$. 
For $X$ of finite type over a non-empty open subscheme $U\subset \Spec(\cO_K)$, 
we have the natural restriction maps {
\[
 KC^{(0)}(X,\nz)[1] \to { KC^{(1)}}(X_K,\nz) 
\to  KC^{(1)}(X_{K_v},\nz) 
\]
}
where $X_L=X\times_U\Spec(L)$ for $L=K$ and $K_v$ (the degree shift $[1]$ comes from
the fact $\Xd {a+1}\cap X_K= \Yd a$ with $Y=X_K$). We define a variant of Kato complex as {
\[
KC ( X/U,\nz )  =\mathrm{cone} [ { KC^{(0)} }(X,\nz )[1] \to \bigoplus_{v\in \Sigma_U} { KC^{(1)}}(X_{K_v} ,\nz ) ],
\]
}
where $\Sigma_U$ denotes the set of $v\in P_K$ which do not correspond to 
closed points of $U$ (thus $\Sigma_U$ includes the set of all infinite places of $K$).

\medbreak

The Kato homology of $X$ (with coefficient in $\nz$) is defined as 
{
\begin{equation}\label{eq.KHintro} 
\KHnz aX =H_a(  KC(X,\nz))  \qfor a\in \bZ.
\end{equation}
}
If $X$ is over $U\subset \Spec(\cO_K)$ as above, we introduce a variant of the Kato homology:
\begin{equation}\label{eq.KHintro2}
\KHccnz a X =H_a(\KCccnz X) \qfor a\in \bZ.
\end{equation}
These are important invariants that reflects arithmetic nature of $X$.
The Hasse principle for the Brauer group of a global field $K$ 
is equivalent to the vanishing of the Kato homology in degree $1$ of the model of $K$.
As a generalization of this fact, Kato proposed the following conjectures called 
the cohomological Hasse principle. 

\begin{conj}\label{KatoconjG} 
Let $X$ be a proper smooth scheme over a finite field. Then 
$$
\KHnz a X =0 \qfor a>0.
$$
\end{conj}
\medbreak

We remark that Geisser \cite{Ge4} defined a Kato complex with integral coefficient
for $X$ over a finite field, and studied an integral version of Conjecture \ref{KatoconjG}.

\begin{conj}\label{KatoconjAL} 
Let $X$ be a regular scheme proper and flat over $\Spec(\cO_k)$ where $\cO_k$ 
is the ring of integers in a local field. Then
$$ \KHnz a X =0 \qfor  a\geq 0. $$
\end{conj}

\begin{conj}\label{KatoconjAG} 
Let $X$ be a regular scheme proper flat over a non-empty open subscheme 
$U\subset \Spec(\cO_K)$ where $K$ is a number field. Then 
$$
\KHccnz a X =0 \qfor a>0.
$$
\end{conj}
\medbreak

Our main results are the following.

\begin{theo}\label{thm1.intro} (Theorem \ref{kato.finitefield})
Conjectures \ref{KatoconjG} and \ref{KatoconjAL} hold if $n$ is invertible on $X$.
\end{theo}

\begin{theo}\label{thm1.intro2} (Theorem \ref{kato.arith})
Conjecture \ref{KatoconjAG} holds if $n$ is invertible on $U$.
\end{theo}

Indeed we will prove the vanishing of the Kato homology with $\qzl$-coefficient
for a fixed prime $\ell$ invertible on $X$, and deduce the vanishing of 
$\lmz$-coefficient by using the Bloch-Kato conjecture (see Section~\ref{finitecoeff}) 
recently established by Rost and Voevodsky 
(the whole proof is available in the papers \cite{V1}, \cite{V2}, \cite{SJ}, \cite{HW}).
\bigskip

We give few words on the known results on the Kato conjectures
(see \cite[Section~4]{Sa3} for a more detailed account).
Let $X$ be as in the conjectures.
The Kato conjectures in case $\dim(X)=1$ rephrase the classical fundamental facts 
on the Brauer group of a global field and a local field.
Kato \cite{K} proved it in case $\dim(X)=2$.  
He deduced it from higher class field theory proved in \cite{KS} and
\cite{Sa1}. For $X$ of dimension $2$ over a finite field, the vanishing of $\KHnz 2 X$ 
in Conjecture \ref{KatoconjG} had been earlier established in \cite{CTSS} 
(prime-to-$p$-part), and by M. Gros \cite{Gr} for the $p$-part
even before Kato formulated his conjectures. 
There have been results \cite{Sa2}, \cite{CT}, \cite{Sw}, \cite{JS1} showing 
the vanishing of the Kato homology with $\qzl$-coefficient in degree$\leq 3$.
In \cite{J2} and \cite{JS2} general approaches to Conjecture \ref{KatoconjG} are proposed 
assuming resolution of singularities. 
Our technique of the proof of Theorem \ref{thm1.intro} is a refinement of 
that developed in \cite{JS2} and it replaces resolution of singularities by
a recent result of Gabber on a refinement of de Jong's alteration (\cite{Il2}) .
Finally we mention the following theorem which will be used in the proof of
Theorem \ref{thm1.intro2}.

\begin{theo}\label{thm.J} (Jannsen)
Let $Y$ be a projective smooth variety over a number field $K$.
Write $Y_{K_v}=Y\times_K K_v$ where $K_v$ is the completion of $K$ at $v\in P_K$.
Then we have a natural isomorphism {
$$
 H_a(KC^{(1)}(Y,\qz))   \cong \underset{v\in P_K}{\bigoplus}\; 
 H_a(KC^{(1)}(Y_{K_v} ,\qz) )\qfor a>0.
$$
}
\end{theo}
\bigskip

We now explain some applications of the main theorem \ref{thm1.intro}.
It turns out that the cohomological Hasse principle plays a significant r\^ole 
in the study of motivic cohomology of arithmetic schemes $X$.
In this paper we adopt
$$\HMXr q=  \CH^r(X,2r-q) .$$
as the definition of the motivic cohomology of regular $X$, where the right hand side
is Bloch's higher Chow group (\cite{B}, \cite{Le} and \cite{Ge1}).
One of the important open problems is the conjecture that motivic cohomology of 
regular arithmetic schemes is finitely generated, a generalization of the 
known finiteness results on the ideal class group and the unit group of 
a number field (Minkowski and Dirichlet), and 
the group of the rational points on an abelian variety over a number field
(Mordell-Weil). 
In \cite{JS2} it is found that the Kato homology fills a gap between 
motivic cohomology with finite coefficient and \'etale cohomology of $X$. 
Indeed, for a regular arithmetic scheme $X$ of dimension $d$
and for an integer $n>0$, there is a long exact sequence
\begin{multline*}\label{eq3.intro}
\KHnz {q+2} X \to \CH^d(X,q;\nz) \rmapo{\rho_{X}} 
H^{2d-q}_{\et}(X,\nz(d)) \\
\to \KHnz {q+1} X \to \CH^d(X,q-1;\nz) \rmapo{\rho_{X}} \dots
\end{multline*}
where $\CH^*(X,q;\nz)$ is Bloch's higher Chow group with finite coefficient and $\rho_X$
is the \'etale cycle map defined in \cite{B}, \cite{GL2}, \cite{Le} and \cite{Sat}.
Thus, thanks to known finiteness results on \'etale cohomology, the cohomological Hasse principle implies new finiteness results on motivic cohomology (see Section~\ref{application1}). 
\medbreak

We also give an implication of Theorem \ref{thm1.intro} on a special value of the zeta function $\zeta(X,s)$ of a smooth projective variety $X$ over a finite field.
It expresses
\[
\zeta(X,0)^*:= \us{s \to 0}{\mathrm{lim}} \
 \zeta(X,s) \cdot (1-q^{-s})
\] 
by the cardinalities of the torsion subgroups of motivic cohomology groups of $X$.
It may be viewed as a geometric analogue of the analytic class number formula for
the Dedekind zeta function of a number field (see Section~\ref{application2}, Remark \ref{rem.zeta}). 
\medbreak

In a forthcoming paper \cite{KeS}, we will give a geometric application of 
the cohomological Hasse principle to singularities. 
A consequence is the vanishing of weight homology groups of the exceptional
divisors of desingularizations of quotient singularities and radicial singularities.

\bigskip

Now we sketch the basic structure of the proof of our main theorem.
We start with an observation in \cite{JSS} that the Kato complex arises
from the niveau spectral sequence associated to \'etale homology theory.
To explain this more precisely, fix the category $\cCS$ of schemes separated and
of finite type over a fixed base $S$. Recall that a homology theory 
$H=\{H_a\}_{a\in \bZ}$ on $\cCS$ is a sequence of functors:
\[
H_a(-):~\cC_S\rightarrow \Mod\quad \text{($\Mod$ is the category of abelian groups)}
\]
which is covariant for proper morphisms and contravariant for open immersion, and
gives rise to a long exact sequence (called the localization sequence)
\[
\cdots\rmapo{\partial}
 H_a(Y) \rmapo{i_*} H_a(X) \rmapo{j^*} H_a(U)\rmapo{\partial} H_{a-1}(Y)
\rmapo{i_*}\cdots
\]
for a closed immersion $i:Y\hookrightarrow X$ and the open complement 
$j:U=X-Y\hookrightarrow X$ where $\partial$ is a connecting homomorphism 
(see Definition \ref{def.homology.theory} for a precise definition).
Given a homology theory $H$ on $\cCS$,
the method of \cite{BO} produces the niveau spectral sequence for $X \in \cCS$:
\begin{equation}\label{spectralsequence.intro}
E^1_{a,b}(X)=\sumd X a H_{a+b}(x)~~\Rightarrow ~~H_{a+b}(X) \qwith 
H_a(x)=\indlim{V\subseteq \overline{\{x\}}} H_a(V).
\end{equation}
Here the limit is over all open non-empty subschemes $V$ of the closure 
$\overline{\{x\}}$ of $x$ in $X$. 
Fundamental examples are the \'etale homology $H = H^{\et}(-,\nz)$
given in Examples \ref{exHK1} and \ref{exHK4}, where $S=\Spec(k)$ for a finite field $k$
or $S=\Spec(\cO_k)$ for the integer ring $\cO_k$ of a local field.
For the spectral sequence \eqref{spectralsequence.intro}
arising from the \'etale homology theory, one has
$\EX a b 1=0$ for $b<-1$ and the Kato complex { $KC^{(0)}(X,\nz)$  }
is identified up to sign with the complex
$$
 \cdots \to \EX a {-1} 1 \rmapo{d^1} \EX {a-1}{-1} 1 \rmapo{d^1} \cdots 
\rmapo{d^1} \EX 1{-1} 1 \rmapo{d^1} \EX 0{-1} 1.
$$
In particular we have a natural isomorphism
$$
\KHnz a X \simeq \EX a {-1} 2
$$
and an edge homomorphism
\begin{equation*}
 \edgehom a X\;:\; \Hetnz {a-1} X \to \KHnz a X 
\end{equation*}
relating the Kato homolgy to the \'etale homology.
\medbreak

Now we start from a general homology theory $H$ on $\cCS$, which satisfies the condition
that for a fixed integer $e$, we have $\EX a b 1=0$ for $X\in \cCS$ and $b<e$ 
(this will be the case if $H$ is leveled above $e$, see Definition~\ref{def.Kato.complex}(1)).
We then define the Kato homology of $X\in \cCS$ associated to $H$ as
$$
\KH a X = \EX a {e} 2\quad (a\in \bZ)
$$
where the right hand side is an $E^2$-term of the spectral sequence 
\eqref{spectralsequence.intro} arising from $H$.
It is an easy exercise to check that the Kato homology 
$$
\KH \empty - =\{\KH a -\}_{a\in \bZ}
$$
provides us with a homology theory on $\cCS$ equipped with an edge homomorphism
\begin{equation}\label{edgehom.intro}
 \edgehom a X\;:\; H_{a+e}(X) \to \KH a X \qfor X\in \cCS
\end{equation}
which is a map of homology theories in an obvious sense 
(see Definition \ref{def.homology.theory}).

In order to prove the Kato conjecture in this abstract setting, we introduce
a condition on $H$, called the {\em Lefschetz condition} (see Definition \ref{def.Lcondition}). 
In case $H$ is the \'etale homology with $\qzl$-coefficient:
\[
\Het(-,\qzl)=\indlim n \Het(-,\lnz),
\]
the Lefschetz condition is shown in \cite{JS3} and \cite{SS}, 
where weight arguments based on Deligne's results \cite{D} play an essential r\^ole.

We now assume that our homology theory $H$ satisfies the Lefschetz condition
and that for a fixed prime $\ell$ invertible on $S$, $H_a(X)$ are $\ell$-primary torsion
for all $X\in \cCS$ and $a\in \bZ$. We also assume $S=\Spec(k)$ for a field $k$ 
(we will treat the case where $S=\Spec(R)$ for a henselian discrete valuation ring
$R$). By induction on $d>0$ we then prove the following assertion (see the sentence above Lemma \ref{mainlem}):

\medskip\noindent
$\mathbf{KC}\mathrm{(} d\mathrm{)}$: For any $X\in \cCS$ with $\dim(X)\le d$ 
projective and smooth over $k$ we have $\KH a X=0$ for $a\ge 1$.

\medskip\noindent
One of the key steps in the proof is to show that under the assupmtion of the Lefschetz
condition, $\mathbf{KC}\mathrm{(} d-1\mathrm{)}$ implies:

\begin{enumerate}
\item[$(*)$]
For $X\in \cCS$ projective, smooth and connected of dimension $d$ over $k$ and for a simple normal crossing divisor $Y\hookrightarrow X$ such that one of the irreducible components is ample,
the composite map
\[
\delta_a\;:\; H_{a+e}(U) \rmapo{\edgehom a U} KH_{a}(U) \rmapo{\partial} KH_{a-1}(Y)
\quad (U=X - Y)
\] 
is injective for $1\le a\le d$ and surjective for $2\le a$, where $\edgehom a U$ is the 
map \eqref{edgehom.intro} and $\partial$ is a connecting homomorphism 
in the localization sequence for the Kato homology.
\end{enumerate}
\medbreak

Now  we sketch a proof of 
$\mathbf{KC}\mathrm{(} d-1\mathrm{)}  \Longrightarrow \mathbf{KC}\mathrm{(} d\mathrm{)}.$
For $X$ as above and for a fixed element $\alpha \in \KH a X$ with $1\leq a\leq d$,
we have to show $\alpha=0$. It is easy to see that there is a dense open subscheme $j:U\to X$ such that: 
\begin{enumerate}
\item[$(**)$]
$j^*(\alpha)$ is in the image of $\edgehom a U: H_{a+e}(U)\to \KH a U$ .
\end{enumerate}
Suppose for the moment that $Y=X - U$ is a simple normal crossing divisor on $X$. Then one can use a Bertini argument to find a hypersurface 
section $H\hookrightarrow X$ such that $Y\cup H$ is a simple normal crossing divisor. Replace $U$ by $U- U\cap H$ and $Y$ by $Y\cup H$ (note that the condition $(**)$ 
is preserved by the modification). Consider the commutative diagram
\[
\xymatrix{
KH_{a+1}(U) \ar[r]^\partial  &   KH_a(Y)  \ar[r]  & KH_a(X)  \ar[r]^{j^*}  &  KH_a(U) \ar[r]^\partial & KH_{a-1}(Y) \\
  H_{a+e+1}(U)  \ar[u]^{\edgehom {a+1} U} \ar[ru]_{\delta_{a+1}}  & & &  H_{a+e}(U)  \ar[u]^{\edgehom a U} \ar[ru]_{\delta_{a}}
}
\]
where the upper row is the exact localization sequence for the Kato homology.
The condition $(*)$ implies that $\delta_a$ is injective and $\delta_{a+1}$ is surjective
by noting that $Y$ is supposed to contain an ample divisor $H\subset X$. 
An easy diagram chase shows that $\alpha =0$. 
\medbreak

In the general case in which $Y\hookrightarrow X$ is not necessarily a simple normal crossing divisor we use a recent refinement of de Jong's alterations due to Gabber (cf.~Remark~\ref{rem.GG}) to find an alteration $f:X'\to X$ of degree prime to { $\ell$} such that the reduced part of $f^{-1}(Y)$ is a simple normal 
crossing divisor on $X'$ which is smooth and projective over $k$. 
We use intersetion theory to construct a pullback map $$f^*:KH_a(X)\to KH_a(X')$$ which allows us to conduct the above argument for $f^*(\alpha )\in KH_a(X')$. 
This means $f^*(\alpha )=0$ and taking the pushforward gives $\deg(f)\, \alpha =0$. 
Since $\deg(f)$ is prime to { $\ell$} and $\alpha$ was supposed to be killed by 
a power of $\ell$, we conclude $\alpha=0$, 
which completes the proof of $\mathbf{KC}\mathrm{(} d\mathrm{)}$.
\medbreak

The most severe technical difficulty which is handled in this paper is the construction of the necessary pullback maps on our homology theories, in particular over a discrete valuation ring. This problem is solved in Section~\ref{pullback} using Rost's version of intersection theory and the method of deformations to normal cones. 
An alternative construction using the Gersten conjecture is given in Section~\ref{pullbacksnd}.

\medbreak  
  
This work is strongly influenced by discussions with Uwe Jannsen. We would like to thank him cordially.


\section{Homology Theory} \label{homology}

\bigskip

\noindent
Let $S$ be the spectrum of a field or of an excellent Dedekind ring. By $\cCS$ we denote the category of schemes $X/S$ which are separated and of finite type over $S$.

The dimension of an integral $X\in \cCS$ is defined to be 
\[
\dim_S (X)= \mathrm{trdeg}(k(X)/k(T)) + \dim(T),
\]
where $T$ is the closure of the image of $X$ in $S$ and $\dim(T)$ is the Krull dimension. 
If all irreducible components $X_i$ of $X$ satisfy $\dim_S(X_i)=d$ we write $\dim_S(X)= d$ and say that $X$ is equidimensional.

For an integer $a\geq 0$, we define $\Xd a$ as the set of all points $x$ on $X$
with $\dimc(\overline{\{x\}})=a$, where $\overline{\{x\}}$ is the closure of $x$ in $X$. One can easily check that
\begin{equation}\label{dimension}
X_a\cap Y =Y_a\quad\hbox{ for $Y$ locally closed in $X$}.
\end{equation}
In the geometric case, i.e.~if $S$ is the spectrum of a field, $x\in X$ belongs to $\Xd a$ if and only if 
$\trdeg_k(\kappa(x)) = a$ where $\kappa(x)$ is the residue field of $x$.

\begin{defi}\label{def.homology.theory}
Let ${\cCS}_\ast$ be the category with the same objects as $\cCS$, but where morphisms are just the proper maps in $\cCS$. 
Let $\Mod$ be the category of abelian groups.
A homology theory $H=\{H_a\}_{a\in \bZ}$ on $\cCS$ is 
a sequence of covariant functors:
$$
H_a(-):~{\cCS}_\ast \rightarrow \Mod
$$
satisfying the following conditions:
\begin{itemize}
\item [$(i)$] For each open immersion $j:V\hookrightarrow X$ in
$\cCS$, there is a map $j^*:H_a(X)\rightarrow H_a(V)$,
associated to $j$ in a functorial way. 
\item [$(ii)$] If $i:Y\hookrightarrow X$ is a closed immersion in $X$, 
with open complement $j:V\hookrightarrow X$, there is a long exact sequence
(called localization sequence)
$$
\cdots\rmapo{\partial} H_a(Y) \rmapo{i_*} H_a(X) \rmapo{j^*} H_a(V) \rmapo{\partial} H_{a-1}(Y)
\longrightarrow \cdots.
$$
(The maps $\partial$ are called the connecting morphisms.) This
sequence is functorial with respect to proper maps or open
immersions, in an obvious way.
\end{itemize}

A morphism between homology theories $H$ and $H'$ is a
morphism $\phi: H \rightarrow H'$ of functors on ${\cCS}_\ast$, 
which is compatible with the long exact sequences in $(ii)$.
\end{defi}

\medbreak

Given a homology theory $H$ on $\cCS$,
we have the spectral sequence of homological type associated to every 
$X \in Ob(\cCS)$, called the niveau spectral sequence (cf. \cite{BO}):
\begin{equation}\label{spectralsequence1}
E^1_{a,b}(X)=\sumd X a H_{a+b}(x)~~\Rightarrow ~~H_{a+b}(X) \qwith 
H_a(x)=\lim_{\lr \atop V\subseteq \overline{\{x\}}} H_a(V).
\end{equation}
Here the limit is over all open non-empty subschemes 
$V\subseteq \overline{\{x\}}$. This spectral sequence
is covariant with respect to proper morphisms in $\cCS$ and
contravariant with respect to open immersions.

\begin{defi}\label{def.Kato.complex}
Fix an integer $e$.
\begin{itemize}
\item[(1)]
Let $H$ be a homology theory on ${\cCS}_\ast$. Then
$H$ is {\it leveled above $e$} if for every affine regular connected $X\in\cCS$ with $d=\dim_S(X)$ and for $a<e$ we have

\begin{equation}\label{KHcondition}
H_{d+a}(X)=0
\end{equation}
\item[(2)]
Let $H$ be leveled above $e$. For $X\in Ob(\cCS)$ with $d=\dim(X)$, define the Kato 
complex of $X$ by
$$
 \KC X \;:\; \EX 0{e} 1\lmapo{d^1}\EX 1{e} 1 \lmapo{d^1} \cdots 
\lmapo{d^1} \EX d{e} 1,
$$
where $\EX a {e} 1$ is put in degree $a$ and the differentials are 
$d^1$-differentials.
\item[(3)]
We denote by $\KH a X$ the homology group of $\KC X$ in degree $a$ and call it the Kato homology of $X$.
By \eqref{KHcondition}, we have the edge homomorphism
\begin{equation}\label{edgehom}
 \edgehom a X\;:\; \Hempty {a+e} X \to \KH a X =\EX a {e} 2.
\end{equation}
\end{itemize}
\end{defi}
\medbreak

\begin{rem}\label{rem.Kato.complex}
\begin{itemize}
\item[(1)]
One easily sees that $\edgehom 0 X$ is always an isomorphism and 
$\edgehom 1 X$ is always surjective.
\item[(2)] 
If $H$ is leveled above $e$, then the homology theory
$\widetilde{H}=H[-e]$ given by $\widetilde{H}_a(X)=H_{a-e}(X)$ for
$X\in Ob(\cC)$ is leveled above $0$.
Thus we may consider only a homology theory leveled above $0$
without loss of generality. 
\end{itemize}
\end{rem}

A proper morphism $f:X\to Y$ and an open immersion $j:V\to X$ induce maps 
of complexes
$$ f_* : \KC X \to \KC Y,\quad j^*: \KC X \to \KC V.$$
For a closed immersion $i:Z\hookrightarrow X$ and its complement 
$j:V \hookrightarrow X$,
we have the following exact sequence of complexes thanks to \eqref{dimension}
\begin{equation}\label{KCexactsequence}
0\to \KC Z \rmapo{i_*} \KC X \rmapo{j^*} \KC V \to 0.
\end{equation}

\bigskip
\def\kb{\overline{k}}
\def\ksep{k^{sep}}

\begin{exam}\label{exHK1}
Assume $S=\Spec(k)$ with a field $k$.
Fix an algebraic closure $\overline{k}$ of $k$ and let $\ksep$ be the separable closure
of $k$ in $\kb$. Let $G_k=\Gal(\ksep/k)$ be the absolute Galois group of $k$.
Fix a discrete torsion $G_k$-module $\Lam$ viewed as a sheaf on $S_{\et}$. 
One gets a homology theory $H = H^{\et}(-,\Lam)$ on $\cCS$:
$$
H^\et_a(X,\Lam)=H^{-a}(X_{\et}, R\,f^{!}\Lam)
\qfor f:X\rightarrow S \text{ in } \cCS.
$$
Here $Rf^!$ is the right adjoint of $Rf_!$ defined in \cite[XVIII, 3.1.4]{SGA4}. 
Sometimes we will write $H^\et_a(X)$ for $H^\et_a(X,\Lam)$ if the coefficient module $\Lam$ is clear from the context. We have the following (see \cite{JSS}).

\begin{lem}\label{lemHK1smooth}
Let $f:X\to S$ be smooth of pure dimension $d$ over $S$.
\begin{itemize}
\item[(1)]
Assume that any element of $\Lam$ is annihilated by an integer prime to $\ch(k)$. 
Then we have 
\begin{equation*}\label{exHK1smooth}
H^\et_a(X,\Lam) = H^{2d-a}(X_{\et},\Lam(d)). 
\end{equation*}
Here, for an integer $r>0$, we put
$$
\Lam(r)=\indlim n f^*\Lam_n\otimes \mu_{n,X}^{\otimes r}\text{  with } \Lam_n=\Ker(\Lam\rmapo{n}\Lam)
$$
where the limit is taken over all integers $n$ prime to $\ch(k)$ and 
$\mu_{n,X}$ is the \'etale sheaf of $n$-th roots of unity on $X$.
\item[(2)]
Assume that $k$ is perfect and any element of $\Lam$ is annihilated by a power of 
$p=\ch(k)>0$. Then we have
\begin{equation*}\label{exHK1smooth}
H^\et_a(X,\Lam) = H^{2d-a}(X_{\et},\Lam(d)). 
\end{equation*}
Here, for an integer $r>0$, we put
$$
\Lam(r)=\indlim n f^*\Lam_n\otimes \DWnlog r X[-r] \text{  with } \Lam_n=\Ker(\Lam\rmapo{p^n}\Lam)
$$
where $\DWnlog r X$ is the logarithmic part of the de Rham-Witt sheaf 
$\DWn r X$ (cf. \cite[I 5.7]{Il1}) and $[-r]$ means a shift in $D^b(X_{\et})$, the derived category of bounded complexes of sheaves on $X_{\et}$.
We remark that we use $\DWnlog r X$ only for $r=\dim(X)$ in this paper.
\end{itemize}
\end{lem}

By Lemma \ref{lemHK1smooth} we get for $X$ general
$$ E^1_{a,b}(X) =\sumd X a H^{a-b}(x,\Lam(a)) 
\qwith H^*(x,\Lam(a))=\lim_{\lr \atop V\subseteq \overline{\{x\}}} H^*(V,\Lam(a)).$$
Here the limit is over all open non-empty subschemes $V\subseteq \overline{\{x\}}$. 
\medbreak

Now assume that $k$ is finite. Due to the fact $\cd(k)=1$
and the affine Lefschetz theorem \cite[XIV, 3.1]{SGA4} together with 
\cite[Lemma 2.1]{Sw}, Lemma \ref{lemHK1smooth} implies
$H^{\et}_{d+a}(X,\Lam) = 0$ for $a<-1$ if $X$ is affine smooth connected of dimension $d$ over $S$. 
Hence $H = H^{\et}(-,\Lam)$ is leveled above $-1$.
The arising complex $\KC X$ is written as:
\begin{multline}\label{KCfinitefield}
\cdots \sumd X a H_{\et}^{a+1}(x,\Lam(a))\to
\sumd X {a-1} H_{\et}^{a}(x,\Lam(a-1))\to \cdots \\
\cdots \to\sumd X 1 H_{\et}^{2}(x,\Lam(1))\to 
\sumd X 0 H_{\et}^{1}(x,\Lam).
\end{multline}
Here the sum over the points $\Xd a$ is placed in degree $a$. 
In case $\Lam=\nz$ it is identified up to sign with complex \eqref{eq.KC1} thanks to \cite{JSS}. 
\end{exam}

\begin{exam}\label{exHK2}
Assume $S=\Spec(k)$ and $G_k$ and $\Lam$ be as in \ref{exHK1}.
If $\Lam$ is finite, we define the homology theory $H^D(-,\Lam)$ by
$$
\HDL a X:=\Hom\big(H^{a}_c(X_{\et},\Lam^\vee),\qz\big) \qfor X\in Ob(\cCS).
$$ 
where $\Lam^\vee=\Hom(\Lam,\qz)$ and $H^{a}_c(X_{\et},-)$ denotes 
the cohomology with compact support over $k$ defined as 
\[
H^{a}_c(X_{\et},-)=H^{a}(\Xb_{\et},-)
\]
for any compactification $j:X\hookrightarrow \Xb$ over $k$, i.e. $\Xb$ is proper over $k$ and $j$ is an open immersion.
If $\Lam$ is arbitrary, we put
$$
\HDL a X=\indlim F H^D_a(X,F) \qfor X\in Ob(\cCS)
$$ 
where $F$ runs over all finite $G_k$-submodules of $\Lambda$.
By Lemma \ref{lemHK1smooth2} below, this homology theory is leveled above $0$.

\begin{lem}\label{lemHK1smooth2}
\begin{itemize}
\item[(1)]
If $X$ is affine regular connected over $S$ with $k$ perfect and $d=\dim(X)$, then
$$
\HDL a X\simeq H^{2d-a}(X_{\overline{k},\et},\Lam(d))_{G_k}\qfor a\leq d,
$$ 
and it vanishes for $a<d$, where $X_{\overline{k}}=X\times_k \overline{k}$ and 
$\Lam(d)$ is defined as in Lemma \ref{lemHK1smooth} and $M_{G_k}$ denotes
the coinvariant module by $G_k$ of a $G_k$-module $M$. 
\item[(2)]
If $k$ is finite, $H^D(-,\Lam)$ shifted by degree $1$ coincides 
with $H^{\et}(-,\Lam)$ in Example \ref{exHK1}
\item[(3)]
If $ k \subset k'$ is a purely inseparable field extension there is a
canonical pushforward isomorphism
\[
H^D(X\otimes_k k' ,\Lambda )  \stackrel{\sim}{\to} H^D(X ,\Lambda ),
\]
where the left hand side is defined as the dual of the cohomology with compact support over $k'$. 
\end{itemize}
\end{lem}
\begin{proof}
We may assume that $\Lam$ is finite.
By the Poincar\'e duality for \'etale cohomology
and \cite[Theorem 2.10 and Remark 2.5 (2)]{JSS}, we have
$$
H^b_c(X_{\overline{k},\et},\Lam^\vee)\simeq
\Hom\big(H^{2d-b}(X_{\overline{k},\et},\Lam(d)),\qz\big),
$$
which vanishes for $b<d$ by the affine Lefschetz theorem \cite[XIV, 3.1]{SGA4} and 
\cite[Lemma 2.1]{Sw}. By the Hochschild-Serre spectral sequence this implies
$H^b_c(X_{\et},\Lam^\vee)=0$ for $b<d$ and 
$H^d_c(X_{\et},\Lam^\vee)\simeq H^d_c(X_{\overline{k},\et},\Lam^\vee)^{G_k}$.
Taking the dual, this implies (1).
(2) follows from the Poincar\'e duality and 
\cite[Theorem 2.10 and Corollary 2.17]{JSS} together with
the Tate duality for Galois cohomology of finite field.
(3) follows from the canonical pullback isomorphism
\[
H^{a}_c(X_{\et},\Lam^\vee) \to H^{a}_c((X\otimes_k k')_{\et},\Lam^\vee) 
\]
which follows from \cite[XVIII 1.2 and XVII 5.2.6]{SGA4}.  
\end{proof}


\end{exam}

\medbreak

In the following example we are in the arithmetic case.

\begin{exam}\label{exHK4}
In this example, we assume $S=\Spec(R)$ where 
$R$ is a henselian discrete valuation ring.
Let $s$ (resp. $\eta$) be the closed (resp. generic) point of $S$.
Let $\cCs\subset \cCS$ be the subcategory of separated schemes of finite type over $s$.
Let $G_S=\pi_1(S,\overline{\eta})$ with a geometric point $\overline{\eta}$ over $\eta$. 
Let $\Lambda$ be a torsion $G_S$-module whose elements are of order prime to $\ch(k)$,
viewed as an \'etale sheaf on $S$.
One has a homology theory $H = H^\et( - ,\Lambda )$ on $\cCS$:
\begin{equation}\label{exHK4def}
H^\et_a( X,\Lambda) = H^{2-a}(X_{\et}, R\,f^{!}\Lambda(1))
\qfor f:X\rightarrow S \text{ in } \cCS,
\end{equation}
where the Tate twist is defined as in Lemma~\ref{lemHK1smooth}(1).
Note that the restriction of $H$ to $\cCs$ coincides with
the homology theory $H^\et(-,\Lambda)$ in Example \ref{exHK1}. The last fact follows from the purity isomorphism
\[
R^a i^! \Lambda(1)=
\left.\left\{\begin{gathered}
 \Lambda \\ 
 0 \\
\end{gathered}\right.\quad
\begin{aligned}
&\text{$a=2$}\\
&\text{$a\not=2$}
\end{aligned}\right.
\]
where $i:s\to S$ is the closed immersion.
By the absolute purity due to Gabber (\cite{FG}, see also \cite[Lemma 1.8]{SS}), 
for $X\in \cCS$ which is integral regular  with $\dimc(X)=d$, we have
\begin{equation}\label{exHK4regular.pnz}
H^\et_a ( X ,\Lambda) = H^{2d-a}_{\et}(X,\Lambda(d)). 
\end{equation}
By \eqref{exHK4regular.pnz} we get for general $X\in \cCS$ 
$$ E^1_{a,b}(X) =\sumd X a H^{a-b}_\et(x,\Lambda(a)).$$
\medbreak

We now assume that $\k(s)$ is finite. Due to the fact $\cd(\k(s))=1$ 
and Gabber's affine Lefschetz theorem \cite[Th\'eor\`eme~2.4]{Il3}, 
\eqref{exHK4regular.pnz} implies that 
$H = H^\et(-,\Lambda)$ is leveled above $-1$.
The arising complex $\KC X$ is written as:
\begin{multline}\label{KC.arith}
\cdots \sumd X a H_{\et}^{a+1}(x,\Lambda(a))\to
\sumd X {a-1} H_{\et}^{a}(x,\Lambda(a-1))\to \cdots \\
\cdots \to\sumd X 1 H_{\et}^{2}(x,\Lambda(1))\to 
\sumd X 0 H_{\et}^{1}(x,\Lambda).
\end{multline}
Here the term $\sumd X a$ is placed in degree $a$. 
For $\Lambda=\Z/n$ this is identified up to sign with complex \eqref{eq.KC1} thanks to \cite{JSS}. 
\end{exam}


\section{Log-Pairs and Configuration complexes} \label{logpairs}

\bigskip

\noindent
In this section we revisit the combinatorial study of the homology of log-pairs originating from~\cite{JS3}.
Let the assumption be as in Section \ref{homology}.
We assume that either
\begin{itemize}
\item[$\bf{(G)}$](geometric case)
$S=\Spec(k)$ for a perfect  field $k$,
\item[$\bf{(A)}$](arithmetic case)
$S=\Spec(R)$ for a henselian discrete valuation ring $R$ with perfect residue field.
\end{itemize}
We denote by $\eta$ (resp.~$s$) the generic (resp.~closed) point of $S$ in the arithmetic case.

\medbreak
\def\cSreg{\cS_{reg}}
\def\cSregT{\cS_{reg/T}}

\def\cSregir{\cS_{reg}^{irr}}
\def\cSregirT{\cS_{reg/T}^{irr}}

\begin{defi}\label{defSreg}
We let $\cSreg\subset \cCS$ be the following full subcategory:
\begin{itemize}
\item 
In the geometric case $X\in \cSreg$ if $X$ is regular and projective over $S$.
\item 
In the arithmetic case $X\in \cSreg$ if $X$ is regular and projective over $S$,
and letting $X^{fl}$ denote the union of those irreducible components of $X$ flat 
over $S$, $X^{fl}_{s,red}$ is a simple normal crossing divisor on $X^{fl}$, where
$X^{fl}_{s,red}$ is the reduced special fibre of $X^{fl}$.
\end{itemize}
We let $\cSregir$ denote the subcategory of $\cSreg$
of the objects which are irreducible.
For a closed subscheme $T\subset S$ ($T=S$, or the closed point of $S$ in the arithmetic case), let $\cSregT$ (resp. $\cSregirT$) denote the subcategory of $\cSreg$ (resp. $\cSregir$) of the objects whose image in $S$ is $T$.
\end{defi}
\medbreak

\begin{defi}\label{defSNCDadm}
For $X\in \cSreg$, a simple normal crossing divisor $Y$ on $X$ is admissible if one of the following conditions is satisfied: 
\begin{itemize}
\item
we are in the geometric case, or
\item
we are in the arithmetic case, and 
letting $Y_1,\dots,Y_r$ be the irreducible components of $Y$ which are flat over $S$,
$Y_1\cup \cdots \cup Y_r\cup X^{fl}_{s,red}$ is a simple normal crossing divisor 
on $X^{fl}$. 
\end{itemize}
\end{defi}

\begin{rem}
In the second case of the above definition, 
\begin{itemize}
\item
$Y\cap X^{fl}$ is a subdivisor of $Y_1\cup \cdots \cup Y_r\cup X^{fl}_{s,red}$,
\item
for all $1\leq i_1<\cdots<i_s\leq r$, 
$Y_{i_1}\cap \cdots \cap Y_{i_s}\;\in \cSreg$ and flat over $S$.
\end{itemize}
\end{rem}

\begin{defi}\label{defS}
We let $\cS\subset \cCS$ be the following full subcategory:
$X\in \cS$ if either $X\in \cSreg$ or there exsits a closed immersion $X\hookrightarrow X'$ such that $X'\in \cSreg$ and $X$ is an admissible simple normal crossing divisor on $X'$. 
\end{defi}

\medbreak

\begin{defi}\label{deflogpair}
\noindent
\begin{itemize}
\item[(1)]
A log-pair is a couple $\Phi=(X,Y)$ where either
\begin{itemize}
\item[-]
$X\in \cSreg$ and $Y$ is an admissible simple normal crossing divisor on $X$,
\item[-]
or $X\in \cS$ and $Y=\emptyset$.
\end{itemize}

\item[(2)]
Let  $\Phi=(X,Y)$ and $\Phi'=(X',Y')$ be log-pairs. 
A map of log-pairs $\pi:\Phi' \to\Phi$ is a  morphism $\pi: X'\to X$ 
such that $\pi(Y')\subset Y$. 
\item[(3)] 
We let $\LP$ denote the category of log-pairs: The objects and morphisms
are as defined in (1) and (2). We have a fully faithful functor:
$$
\cS\to \LP\;;\; X\to (X,\emptyset).
$$
For a log-pair $\Phi = (X,Y) $ with $X\in \cSreg$ we let $\hat \Phi$ be the log-pair $(Y,\emptyset )$. This defines a functor $\LP\to \LP$.
\item[(3)]
Let $\Phi=(X,Y)$ be a log-pair with $X\in \cSreg$.
Let $Z\hookrightarrow X$ be a closed immersion in $\cS$ such that
$\Phi_Z=(Z,Z\cap Y)$ and $\Psi_Z=(X,Y\cup Z)$ are log-pairs.
Then the sequence in $\LP$
$$\Phi_Z \to \Phi \to \Psi_Z$$
is called a fiber sequence in $\LP$, where the first (resp. second)
map is induced by the closed immersion $Z\to X$ (resp. the identity on $X$).
\item[(4)]
Let $\Phi=(X,\emptyset)$ be a log-pair.
Let $X_1$ be an irreducible component of $X$ and $X_2$ be the union of the other
irreducible components of $X$. 
Then the sequence in $\LP$
\[
\Phi_{X_1\cap X_2} \to \Psi \to \Phi
\]
is also called a fiber sequence. Here 
$\Phi_{X_1\cap X_2} = (X_1\cap X_2,\emptyset)$ and  
$\Psi= (X_1 \coprod X_2 ,\emptyset)$ (by definition these are log-pairs).
\end{itemize}
\end{defi}
\medbreak
\def\CAb{C^b(\Mod)}
\def\KAb{K^b(\Mod)}

Next we present a construction which allows us to extend certain functors defined a priori only on $\cSreg$ to all log-pairs.
It is an elementary version of the construction of Gillet and Soul\'e~\cite{GS}.
Let $\CAb$ be the category of homologically  bounded complexes of abelian groups.
Assume given a covariant functor $F: \cSreg\to \CAb$ such that the natural maps give an isomorphism
\begin{equation}\label{log.sumsum}
F(X_1) \oplus F(X_2)\overset{\cong}{\rightarrow}  F(X_1\coprod X_2) .
\end{equation}
We now construct in a canonical way a functor $\bar F:\LP\to \CAb$ which satisfies:
\begin{itemize}
\item $\bar F$ maps fibre sequences of log-pairs to fibre sequences in the derived category. 
\item For $X\in \cSreg$ and $\Phi=(X,\emptyset)$,
$\bar F ( \Phi )=F(X)$.
\end{itemize} 
\medbreak

For a simplicial object in $\cSreg$
\[
X_\bullet\;:\; \quad
\cdots \;X_2 \;
\begin{matrix}
\rmapo{\delta_0}\\
\lmapo{s_0}\\
\rmapo{\delta_1}\\
\lmapo{s_1}\\
\rmapo{\delta_2}\\
\end{matrix}
\; X_1 \;
\begin{matrix}
\rmapo{\delta_0}\\
\lmapo{s_0}\\
\rmapo{\delta_1}\\
\end{matrix}
\; X_0 ,
\]
let
$\text{Tot} \, F(X_\bullet )\in \CAb$
be the total complex associated to the double complex
\[
\cdots\to F(X_n) \rmapo{\partial} F(X_{n-1})\rmapo{\partial} \cdots
\rmapo{\partial} F(X_0),
\]
where $\displaystyle{\partial=\sum_{a=0}^n (-1)^a (\delta_a)_*}$. 

For a log-pair $\Phi=(X,Y)$ where $X\in \cSreg$ and $Y$ is an admissible 
simple normal crossing divisor on $X$, we set
\[
\bar F ( \Phi )= \text{Tot} \, F((\coprod_{j\in J} Y_j\to X)_\bullet ) [-1]
\]
where $Y_j\hr Y$ ($j\in J $) are the irreducible components of $Y$ and
$(\coprod_{j} Y_j\to X)_\bullet$ is the augmented \v Cech simplicial scheme associated 
to the map $\coprod_j Y_j\to X$, 
which is easily seen to be a simplicial object in $\cSreg$.

For a log-pair $\Phi=(Y,\emptyset)$ where $Y\in \cS$ is an admissible 
simple normal crossing divisor on $X\in \cSreg$, we set
$\bar F ( \Phi )$ to be the mapping fiber of the natural map
\[
F(X)  \to \text{Tot} \, F((\coprod_{j\in J} Y_j\to X)_\bullet ) [-1] 
\]
It is easy to check that the functor $\bar F$
satisfies the properties stated above.

\medskip

In what follows we fix a homology theory $H$ on $\cCS$ leveled above $e$ as in 
Definition \ref {def.Kato.complex}. 
We denote the restriction of the Kato complex functor $KC_H$ to $\cSreg$ by the same letter. The construction explained above produces a functor 
\[
\overline{KC }_H \;:\; \LP \to \CAb.
\]
For simplicity of notation we will omit the bar and write  
$KC_H(\Phi)=\overline{KC }_H (\Phi) $.
We easily see the following lemma.

\begin{lem}\label{lem1.Kato.complex}
For a log-pair $\Phi=(X,Y;U)$ with $X\in \cSreg$,
there is a natural quasi-isomorphism
\begin{equation*}\label{KCFU}
KC_H ({\Phi}) \isom \KC U 
\end{equation*}\
compatible with fibre sequences of log-pairs.
\end{lem}

\medbreak

We need another type of a Kato complex below.
The reduced Kato complex of a log-pair $\Phi=(X,Y)$ is defined 
to be the complex $KC_H(\hat \Phi )$, where we recall $\hat \Phi=(Y,\emptyset)$.
For later reference we state the following variant of Lemma~\ref{lem1.Kato.complex}.

\begin{lem}\label{lem1.reduced.Kato.complex}
For a log-pair $\Phi=(X,Y)$ with $X\in \cSreg$, there is a natural quasi-isomorphism
\begin{equation*}\label{KCFU}
\KC {\hat \Phi} \isom \KC Y .
\end{equation*}
\end{lem}
\bigskip

Next we define the trace map for a Kato complex.
For $X\in \cSregir$ we put 
\begin{equation}\label{eq.LamX}
\Lambda_H (X):= \KC {T_X},\text{ where $T_X$ is the image of $X$ in $S$}.
\end{equation}
The functor $\Lam_H:\cSregir\to \CAb$ extends naturally to $\cSreg$ 
via~\eqref{log.sumsum}. 
Now apply the bar construction to $\Lam_H$ to get a functor 
\[
\overline{\Lam_H}:\LP \to \CAb.
\]
For simplicity of notation 
we write $\Lam_H(\Phi)=\overline{\Lam_H}(\Phi)$ for a log-pair $\Phi$ and call it
the {\em configuration complex} of $\Phi$ (with coefficient $\Lam_H$) . We set
\[
G_a(\Phi) := H_a( \Lam_H(\Phi)) .
\]
We also write $G_a(X)=G_a(\Phi)$ for a log pair $\Phi=(X,\emptyset)$.

\medbreak
For $X\in \cSregir$ with image $T_X$ in $S$ we have the trace map
\begin{equation}\label{KCtrace0}
\KC X \to \KC {T_X}=\Lam_H(X)
\end{equation}
induced by the proper map $X\to T_X$. By the bar construction, this extends 
to a trace map of complexes for a log-pair $\Phi$: 
\begin{equation}\label{KCtrace1}
 \tr_\Phi : \;\KC \Phi \to \Lam_H(\Phi)
\end{equation}
which is a natural transformation of functors $\LP\to \CAb$.

If $\Phi=(X,Y;U)$ we can use Lemma~\ref{lem1.Kato.complex} to construct a morphism
\begin{equation}\label{eq.trU}
\tr_{U}: KC_H(U) \simeq KC_H( \Phi ) \lr \Lam_H(\Phi )
\end{equation} 
in the derived category. This induces a homomorphism of homology groups
\begin{equation}\label{graphhomU}
\graphhom {a}  {\Phi} \;:\; \KH a U \to 
\graphHempty a {\Phi}
\end{equation}
By the Gillet-Soul\'e construction explained above and Lemma \ref{lem1.Kato.complex}, 
the following diagram is commutative
$$
\begin{CD}
 \KH a U @>{\partial}>> \KH {a-1} Y \\
 @VV{\graphhom a \Phi}V @VV{\graphhom {a-1} \hPhi}V \\
 \graphHempty a {\Phi}@>{\partial}>>\graphHempty {a-1} {\hPhi}\\
\end{CD}
$$
where $\partial$ is the boundary map arising from 
the fibre sequence
\[
\hat \Phi \lr  (X,\emptyset ) \lr \Phi .
\]
This fibre sequence implies:

\begin{lem}\label{log.rediso}
For a closed subscheme $T\subset S$, let $i_T\in \{-\infty,0,1\}$ be
defined as follows. If $H_a(\KC T)=0$ for all $a\in \bZ$, $i_T=-\infty$.
Otherwise 
\[i_T=\max\{a\in \bZ\;|\; H_a(\KC T)\not=0\}.\]
For a log-pair $\Phi=(X,Y)$ with $X\in \cSregirT$, 
the boundary map
\[
G_a(\Phi) \stackrel{\partial}{\lr} G_{a-1}(\hat \Phi )
\]
is an isomorphism for $a\ge i_T+2$ and injective for $a=i_T+1$.
\end{lem}

\begin{lem}\label{log.vanishing}
For a log-pair $\Phi=(X,Y)$, $G_a(\Phi)=0$ for $a>\dim_S(X)$.
\end{lem}
\begin{proof}
Evident from the definition.
\end{proof}

\begin{lem}\label{lem.reduced.graphhom}
Let $\Phi=(X,\emptyset)$ be a log-pair and $q\geq 0$ be an integer.
Assume that for any $Z\in \cSregir$ with $\dim_S(Z)\le \dim_S(X)$ 
the map \eqref{KCtrace0} induces a quasi-isomorphism
$$\tau_{\leq q} \KC Z \isom \tau_{\leq q}\Lam_H(Z).$$
Then $\tr_\Phi $ induces a quasi-isomorphism
$ \tau_{\leq q} \KC \Phi \isom \tau_{\leq q}\Lam_H(\Phi)$.
\end{lem}

\begin{proof}
We prove the lemma by a double induction on the dimension of $X$ and on the number of irreducible components of $X$. Let $X_1$ be an irreducible component of $X$ and
let $X_2$ be the union of the other irreducible components.
The fibre sequence
\[
\Phi_{X_1\cap X_2} \to \Psi \to \Phi
\]
with $\Phi_{X_1\cap X_2}=(X_1\cap X_2,\emptyset)$ and $\Psi = (X_1 \coprod X_2,\emptyset)$ gives rise to a morphism of fibre sequences
in the derived category
\[
\xymatrix{
KC_H( \Phi_{X_1\cap X_2} ) \ar[r]\ar[d]  & KC_H( \Psi) \ar[r] \ar[d]  &  KC_H(\Phi)  \ar[d]\\
\Lam_H( \Phi_{X_1\cap X_2})  \ar[r]  &  \Lam_H(\Psi)  \ar[r]   &  \Lam_H(\Phi)  
}
\] 
where the first and second vertical arrow induce isomorphisms on homology groups
in degree $\leq q$. So the last arrow induces isomorphisms on homology groups
in degree $\leq q$. 
\end{proof}


\section{Lefschetz Condition} \label{lefschetz}

\bigskip

\noindent
Let the notation be as in the previous section. We introduce the
Lefschetz condition which will be crucial in the proof of our main theorem in the
next section. Roughly, the Lefschetz condition for a homology theory $H$ says that
for a log-pair $(X,Y;U)$ such that an irreducible component of $Y$ is ample, we 
can calculate $H_{a+e}(U)$ for $a\leq \dim(U)$ by using the configuration complex
introduced in the previous section. Below we explain and generalize the arguments given 
in~\cite[Lemma 3.4]{JS2} and~\cite{SS}, which show
that the Lefschetz condition is satisfied for
the homology theories (with admissible coefficients, see Definition \ref{admisdefi}) introduced in Section~\ref{homology}. The use of weight arguments is pivotal.   
\medbreak

For a log-pair $\Phi=(X,Y)$ with $X\in \cSreg$, let
\begin{equation}\label{graphedgehomPhi}
\graphedge a \Phi \;:\; \Hempty {a+e} U \to 
\graphHempty a {\Phi}
\end{equation}
be the composition of the maps $\graphhom {a}  {\Phi}$ and $\edgehom a U$ which were defined in
 \eqref{graphhomU} and \eqref{edgehom}.

\begin{defi}\label{def.Lcondition}
Let $\Phi=(X,Y)$ be a log-pair with $X\in \cS^{irr}_{reg}$ and let $T_X$ be its image in $S$. 
\begin{itemize}
\item[(1)]
$\Phi$ is $H$-clean in degree $q$ for an 
integer $q$ if $q\leq \dim_S(X)$ and $\graphedge a \Phi$ is injective for $a=q$ 
and surjective for $a=q+1$.
\item[(2)] Assume $X$ is irreducible. Then $\Phi$ is ample 
if $\dim_S(X)> \dim_S(T_X)$ and there exists a regular closed subscheme $Y'$ of $Y$ such that $Y'$ is a divisor on $X$ relatively ample over $T_X$.
\item[(3)]
We say that $H$ satisfies the Lefschetz condition if 
a log-pair $\Phi=(X,Y)$ is $H$-clean in degree $q$ for all $q \leq \dim_S(X)$ 
whenever $\Phi$ is ample or $\dim_S(X)=\dim_S(T_X)$. 
\end{itemize}
\end{defi}

The following Bertini theorem (\cite{P},\cite{Ga},\cite{JS3}) shows that there are plenty of ample log-pairs.

\begin{theo}\label{Bertini}
Let $\Phi=(X,Y)$ be a log-pair, where $X\in \cSregir$ with $\dim_S(X)> \dim_S(T_X)$.
Fix an invertible sheaf $\mathcal L$ on $X$ relatively ample over $T_X$. 
For a sufficiently large $N>0$ there exists a section of $\mathcal L^{\otimes N}$
with support $Z\hr X$ such that $\Phi'=(X,Y\cup Z)$ is an ample log-pair.
\end{theo}
\medbreak

\begin{lem}\label{lem0.Lefschetz}
Assume that $H$ satisfies the Lefschetz condition. Let $T\subset S$ be a closed
subscheme. For $X\in \cSregirT$ the map
$\KH a X \to \KH a T$ induced by the proper map $X\to T$ 
is an isomorphism for $a\leq \dim_S(T)$.
\end{lem}
\begin{proof}
By Remark \ref{rem.Kato.complex}(1) it suffices to show that the composite map
\[
H_{a+e}(X) \rmapo{\edgehom a X} \KH a X \to \KH a T
\]
is an isomorphism for $a\leq \dim_S(T)$.
Note that the above map is $\graphedge a \Phi$ for $\Phi=(X,\emptyset)$. Hence, 
the assertion in case $\dim_S(X)=\dim_S(T)$ follows from the Lefschetz condition.
In case $\dim_S(X)>\dim_S(T)$, use the Bertini Theorem~\ref{Bertini} to choose 
 regular divisor $Y\hr X$ relatively ample over $T$ such that $\Phi=(X,Y;U)$
forms an ample log-pair. From the fiber sequence of log-pairs:
\[
\Phi_Y \to\Phi_X  \to \Phi,\text{  where } \Phi_Y=(Y,\emptyset),\; \Phi_X=(X,\emptyset),
\]
we get the following commutative diagram with exact rows 
\[
\xymatrix{
H_{a+1+e}(U) \ar[r] \ar[d]_{\simeq} & H_{a+e}(Y)  \ar[d]_{\simeq} \ar[r]  &  H_{a+e}(X)   \ar[r] \ar[d]  & H_{a+e}(U)  \ar[d]_{\simeq}  \ar[r] & H_{a-1+e}(Y)\ar[d]_{\simeq} \\
G_{a+1}(\Phi)  \ar[r]  &  G_a(\Phi_Y)   \ar[r]  &   G_a(\Phi_X)  \ar[r]  &  G_a(\Phi)  \ar[r]  & G_{a-1}(\Phi_Y)
}
\]
where for $a\leq \dim_S(T)$, the indicated maps are isomorphisms by the Lefschetz condition and the induction assumption (note $G_a(\Phi_X)=G_a(\Phi_Y)=\KH a T$). 
So the five lemma shows that the middle vertical arrow
is an isomorphism too. This proves the lemma.
We note that the above diagram is commutative since the vertical maps are the compositions of the maps induced by the natural transformation from
the homology theory $H$ to the associated Kato homology $KH$ and the maps induced by \eqref{eq.trU}.
\end{proof}

\bigskip

In the rest of this section we assume that the base scheme $S$ is the
spectrum of the perfection $k$ of a finitely generated field or the
spectrum of a henselian discrete valuation ring $R$ with finite
residue field. Let { $\ell$} be a prime number. The homology theory
$H$ will be as in one of the examples of Section~\ref{homology} and we
assume $k$ is finite if $H=H^\et$ (geometric case) and { $\ell$} is
invertible in $R$ if $H=H^\et$ (arithmetic case).

Let $G=\pi_1(S,\overline{\eta})$ with a geometric generic point $\overline{\eta}$. 
The following definition is motivated by \cite{J2}.

\begin{defi}\label{admisdefi}
A finitely generated free { $\Z_{\ell}$}-module $T$ with trivial $G$-action is called  
admissible of weight $0$. 
A finitely generated free { $\Z_{\ell}$}-module $T$ with continuous $G$-action is called admissible if there is an exact sequence
of free { $\Z_{\ell}$}-modules with continuous $G$-actions 
\[
0\lr T' \lr T \lr T'' \lr 0
\]
such that $T'\otimes_{\Z_\ell} \Q_\ell$ is mixed of weights $<0$ \cite{D} and $T''$ is admissible of weight $0$.
A torsion $\Z_\ell$-module $\Lam$ with $G$-action is called admissible if there is an admissible free $\Z_\ell$-module $T$ such that
$\Lam=T\otimes_{\Z_\ell} \Q_\ell/\Z_\ell$.
\end{defi}

The next theorem comprises results due to Jannsen--Saito \cite{JS2} and Saito--Sato \cite{SS}.

\begin{theo}\label{thm.Lcondition}
Let $H_*(-,\Lam)$ be one of the homology theories defined in Section~\ref{homology} and $S$ as explained above. Assume 
$\Lambda=T\otimes_{\Z_\ell} \Q_\ell/\Z_\ell$ is an admissible torsion $G$-module as in Definition~\ref{admisdefi}.
Then $H$ satisfies the Lefschetz condition.
\end{theo}

\begin{proof} 
Take an ample log-pair $\Phi=(X,Y;U)$ with $X\in \cSreg^{irr}$. We want to show
the map 
\begin{equation}\label{eq0.thm.Lcondition}
\graphedge a \Phi \;:\; \Hempty {a+e} U \to 
\graphHempty a {\Phi}
\end{equation}
is an isomorphism for $a\leq\dim_S(X)$. 

First we treat the homology theory $H=H^D(-,\Lam)$ of Example~\ref{exHK2} over the base
scheme $S=\Spec(k)$. Recall that $H$ is leveled over $e=0$.
It follows immediately from the definition that for $X\in \cSreg$
\begin{equation}\label{eq1.thm.Lcondition}
\HDL 0 X \simeq \HDL 0 S=\Lam_G,\text{  and  } \Lam_H(X)= \Lam_G[0],
\end{equation}
where $\Lam_G$ is the coinvariant of $\Lam$ under $G$.
We now distinguish several cases. 
For $X\in \cCS$ smooth over $k$, put $X_{\overline{k}}=X\times_k \overline{k}$ with
an algebraic closure $\overline{k}$ of $k$ and write
\def\Xkb{X_{\overline{k}}}
$$
H^i(X_{\overline{k}},\qzl(r))=
\left\{\begin{gathered}
 \indlim n H^i(X_{\overline{k},\et},\mu_{\ell^n}^{\otimes r}) \quad \text{$\ell\not=\ch(k)$}\\
 \indlim n H^{i-r}(X_{\overline{k},\et},\DWnlog r {\Xkb}) \quad \text{$\ell=p:=\ch(k)$}\\
\end{gathered}\right.
$$
$$
H^i(X_{\overline{k}},\ql(r))=
\left\{\begin{gathered}
 \projlim n H^i(X_{\overline{k},\et},\mu_{\ell^n}^{\otimes r})\otimes_{\zl}\ql 
\quad \text{$\ell\not=\ch(k)$}\\
 \projlim n H^{i-r}(X_{\overline{k},\et},\DWnlog r {\Xkb})\otimes_{\zp}\qp
\quad \text{$\ell=p:=\ch(k)$}\\
\end{gathered}\right.
$$
\medbreak

{\bf Case I-1:} 
The desired assertion in case $\dim_S(X)=0$ follows immediately from \eqref{eq1.thm.Lcondition}. Assume $\dim_S(X)=1$. 
In view of \eqref{eq1.thm.Lcondition} we are reduced to show that the sequence
\begin{equation}\label{lef.eqex}
0\lr \HDL 1 U \rmapo {\partial} \bigoplus_{y\in \Yd 0} \HDL 0 y \lr \HDL 0 X \lr 0
\end{equation}
is exact. For this we may assume without loss of generality that $X$ is 
geometrically connected over $k$. The assertion is clear except the injectivity 
on the left. An easy computation shows  
\begin{equation}\label{lef.eq2}
\bigoplus_{y\in Y} \HDL 0 y = \big( \bigoplus_{y\in Y_{\bar k}} \Lam \big)_G 
 = \big( \bigoplus_{y\in Y_{\bar k}} T\otimes_{\zl} \qzl\big)_G 
\end{equation}
By the reason of weight, the exact sequence $0\to T'\to T\to T''\to 0$ in \ref{admisdefi} induces an isomorphism
\[
\big( \bigoplus_{y\in Y_{\bar k}} \Lam \big)_G =
\big( \bigoplus_{y\in Y_{\bar k}}  T''\otimes_{\Z_\ell} \Q_\ell/\Z_\ell \big)_G.
\]
On the other hand Lemma \ref{lemHK1smooth2} implies
\begin{equation}\label{eq2.thm.Lcondition}
\HDL 1 U = H^1(U_{\bar k,\et} , \Lam(1))_G= \big(H^1(U_{\bar k},\qzl(1))\otimes_{\zl} T)_G
\end{equation}
We claim that the exact sequence in \ref{admisdefi} induces an isomorphism
$$
 H^1(U_{\bar k,\et} , \Lam(1))_G
\simeq \big(H^1(U_{\bar k},\qzl(1))\otimes_{\zl}  T''\big)_G \;.
$$
By the claim we can assume without loss of generality that $T=T''$ is admissible of 
weight $0$ so that $\Lam=(\qzl)^r$ with trivial $G$-action.
To show the claim, it suffices to prove $\big(H^1(U_{\bar k},\qzl(1))\otimes T'\big)_G=0$.
The affine Lefschetz theorem (case $\ell\not=\ch(k)$) or \cite[Lemma 2.1]{Sw} (case $\ell=\ch(k)$)
implies that the natural map
$$
H^1(U_{\bar k},\ql(1)) \to H^1(U_{\bar k},\qzl(1))
$$
is surjective and the claim follows from the vanishing of 
$\big(H^1(U_{\bar k},\ql(1))\otimes_{\zl} T'\big)_G$. 
We have the exact localization sequence
\[
 H^1(X_{\bar k} ,\Q_\ell(1)) \lr H^1(U_{\bar k} ,\Q_\ell(1)) \lr 
H^0(Y_{\bar k} ,\Q_\ell)
\]
where we use \cite[Theorem 2.5]{Sw} in case $\ell=\ch(k)$.
It shows that $H^1(U_{\bar k} ,\Q_\ell(1))$ is mixed of weight $-1$ and $0$,
thanks to \cite{D} in case $\ell\not=\ch(k)$ and to \cite[p784 $(iii)$]{CTSS} and 
\cite[Theorem 1.3]{GrSw} in case $\ell=\ch(k)$. Since $T'$ is mixed of weight $<0$,
this implies the desired vanishing.
\medbreak

Now we show the injectivity of $\partial$ in \eqref{lef.eqex} assuming
$\Lam=(\qzl)^r$ with trivial $G$-action.
Consider the exact localization sequence
\begin{equation}\label{lef.eq3}
H^1(X_{\bar k,\et} , \Lam(1)) \lr  H^1(U_{\bar k,\et} , \Lam(1)) \lr \bigoplus_{y\in Y_{\bar k}} \Lam  \lr \Lam \lr 0 .
\end{equation}
By the weight argument as before, we get 
\begin{equation}\label{eq3.thm.Lcondition}
(H^1(X_{\bar k,\et} , \Lam(1)))_G=0,
\end{equation}
Bt \eqref{lef.eq2} and \eqref{eq2.thm.Lcondition},
the sequence~\eqref{lef.eq3} would give us the desired injectivity, 
if we showed that it stays exact after taking coinvariants.  
Let $M$ be defined by the short exact sequence
\begin{equation}\label{eq4.thm.Lcondition}
0\lr M \lr \bigoplus_{y\in Y_{\bar k}} \Lam  \lr \Lam \lr 0  .
\end{equation}
From \eqref{lef.eq3} and \eqref{eq3.thm.Lcondition},
we deduce that $H^1(U_{\bar k,\et} , \Lam(1))_G=M_G$. 
Let $N\subset G$ be an open normal subgroup which acts trivially on the set $Y_{\bar k}$. 
Since $\Lam=(\qzl)^r$ with trivial $G$-action, $N$ acts trivially on each term of 
\eqref{eq4.thm.Lcondition} and it splits as a sequence of $N$-modules.
Hence we get the exact sequence
\[
H_1(G/N, \Lam ) \lr H^1(U_{\bar k,\et} , \Lam(1))_G \lr 
(\bigoplus_{y\in Y_{\bar k}} \Lam)_G  \lr \Lam_G \lr 0
\]
Since $\Lam=(\qzl)^r$ is divisible, we have $H_1(G/N, \Lam )=0$, 
which finishes the proof of Case I-1.
\medbreak

{\bf Case I-2:}
Assume $d=\dim_S(X)>1$ and $Y$ is irreducible. 
In this case it is clear from \eqref{eq1.thm.Lcondition}
that $G_a(\Phi)=0 $ for $a\in \Z$. 
On the other hand Lemma \ref{lemHK1smooth2} implies that
\begin{equation}\label{lef.eq4}
\HDL a U = H^{2d-a} (U_{\bar k,\et} ,\Lam(d) )_G \text{  for } a\le d
\end{equation}
and it vanishes for $a<d$. Hence it remains to
show vanishing in~\eqref{lef.eq4} for $a=d$. 
The affine Lefschetz theorem  (case $\ell\not=\ch(k)$) or \cite[Lemma 2.1]{Sw} (case $\ell=\ch(k)$)
implies
\[ H^d(U_{\bar k} ,\Q_\ell(d)) \lr H^d(U_{\bar k} ,\Q_\ell/\Z_\ell(d)), \]
is surjective. Hence it suffices to show $H^d(U_{\bar k} ,\Q_\ell (d))_G=0$.
We have the exact localization sequence
\[
\cdots \to H^d(X_{\bar k} ,\Q_\ell (d)) \lr H^d(U_{\bar k} ,\Q_\ell (d)) \lr H^{d-1}(Y_{\bar k} ,\Q_{\ell}(d-1))\lr \cdots
\]
where we use \cite[Theorem 2.5]{Sw} in case $\ell=\ch(k)$. 
It shows that $H^d(U_{\bar k} ,\Q_{\ell}(d))$ is mixed of weight $-d$ and $-d+1<0$
thanks to \cite{D} in case $\ell \not=\ch(k)$ and to \cite[p784 $(iii)$]{CTSS} and 
\cite[Theorem 1.3]{GrSw} in case $\ell=\ch(k)$. 
This implies the desired vanishing and the proof of Case I-2 is complete.
\medbreak

{\bf Case I-3:} Assume $d=\dim_S(X)>1$ and that $Y$ consists of $r>1$ irreducible components $(Y_j)_{1\le j\le r}$. Let $Y_r$ be ample. We use a double induction on the dimension $d$ of $X$ and on the number $r$ of irreducible components of $Y$, which will reduce us to Case I-1 and Case I-2. Consider the fibre sequence
\[
\Phi_{Y_1} \to \Psi  \to \Phi
\]
where $\Phi_{Y_1}=(Y_1, Y_1\cap ( \cup_{j>1} Y_j )\, ;W )$ and 
$\Psi=(X,\cup_{j>1} Y_j \, ; V )$.
Writing $H_a(Z)=\HDL a Z$ for $Z\in \cCS$ for simplicity, 
it gives rise to the following commutative diagram with exact rows
\[
\xymatrix{
H_a(W)  \ar[r]  \ar[d]  &  H_a(V)  \ar[r]  \ar[d]_{\simeq}  &  H_a(U)   \ar[r]  \ar[d]  &   H_{a-1}(W)  \ar[r]  \ar[d]_{\simeq}  & H_{a-1}(V) \ar[d]_{\simeq} \\
G_a(\Phi_{Y_1} )  \ar[r]  &  G_a(\Psi)  \ar[r]  &  G_a(\Phi)  \ar[r]  &  G_{a-1}( \Phi_{Y_1} ) \ar[r]  & G_{a-1}( \Psi )
}
\]
The commutativity of the diagram follows by the same argument as in the last part of the proof of Lemma \ref{lem0.Lefschetz}. 
For $a\le d$, the indicated maps are isomorphisms by our induction assumption and
the left vertical map is surjective by our induction assumption and
Lemma~\ref{log.vanishing}. Hence a diagram chase shows that the middle
vertical arrow is bijective. This completes the proof of Theorem \ref{thm.Lcondition}
for the homology theory in Example~\ref{exHK2}.
It implies the theorem for the homology theory in Example~\ref{exHK1} by 
Lemma \ref{lemHK1smooth2}(3).
\medbreak

Finally we prove the theorem for the homology theory $H = H^\et( - ,\Lambda )$
in Example~\ref{exHK4} over the base scheme $S=\Spec(R)$ where $R$ is a henselian 
discrete valuation ring with finite residue field.
Recall that this is leveled over $e=-1$.
Note that the restriction of $H$ on $\cCs$ coincides with
the homology theory in Example \ref{exHK1}, where
$\cCs\subset \cCS$ is the subcategory of separated schemes of finite type over 
the closed point $s$ of $S$. This allows us to deal with \eqref{eq0.thm.Lcondition}
only in case $X$ is flat over $S$ (so that $T_X=S$). We now distinguish some cases.
Recall that $\Lam$ is assumed to be $\ell$-primary torsion with $\ell$ invertible on $S$. 
\medbreak

{\bf Case II-1:} 
Assume $\dim_S(X)=1$.
In this case $X=\Spec(R')$ where $R'$ is a henselian discrete valuation ring
which is finite over $R$. Let $s'$ (resp. $\eta'$) be the closed (resp. generic) 
point of $X$. 
Then the log-pair $\Phi=(X,Y)$ is equal to either $(X,\emptyset)$ or $(X,s';\eta')$.
By definition, $\Lam_H(X)=\KC S$ is the complex:
\[
\begin{matrix}
H^1(s_{\et},\Lam) &\lmapo{\partial}& H^2(\eta_{\et},\Lam(1)) \\
\text{deg $0$}&& \text{deg $1$}&&
\end{matrix}
\]
which is acyclic due to the localization exact sequence and the isomorphism
\[
H^a(S_{\et},\Lam(1))\simeq H^a(s_{\et},\Lam(1))=0\text{  for } a\geq 2.
\]
In case $\Phi=(X,\emptyset)$, this implies $G_a(\Phi)=H_a(\KC S)=0$ for all 
$a\in \bZ$. On the other hand we have
\[
\HetLam {a-1} X\simeq H^{3-a}(X_{\et},\Lam(1))\simeq H^{3-a}(s'_{\et},\Lam(1)).
\]
The last group vanishes for $a\leq 1$ because of the fact $\cd(s')=1$.
In case $\Phi=(X,s';\eta')$, $\Lam_H(\Phi)$ is equal to the complex
$H^2(\eta_{\et},\Lam(1))$ put in degree $1$, while
\[
\HetLam {a-1} {\eta'} \simeq H^{3-a}(\eta'_{\et},\Lam(1))
\]
which vanishes for $a=0$. For $a=1$, the map \eqref{eq0.thm.Lcondition} is 
identified with the correstriction map
\[
H^{2}(\eta'_{\et},\Lam(1))  \to H^{2}(\eta_{\et},\Lam(1))
\]
which is known to be an isomorphism. This proves the desired assertion in this case.
\medbreak

{\bf Case II-2:} 
Assume $d:=\dim_S(X)>1$ and $Y$ is a regular flat and relatively ample divisor over $S$.
It is clear that $G_a(\Phi)=0$ for $a\in \Z$.
Using duality \eqref{exHK4regular.pnz} and ba ase change theorem \cite[Lemma 3.4]{SS}, 
we conclude that
\begin{equation}\label{lef.eq5}
H_{a-1}^\et (U,\Lam)\simeq H^{2d-a+1}(U_\et , \Lam (d) ) 
\simeq H^{2d-a+1}(U_{s,\et} ,\Lam (d) )
\end{equation}
By the assumption of ampleness of $Y$, $U_s$ is affine of dimension $d-1$ over $s$.
Hence the last group of~\eqref{lef.eq5} vanishes for $a\le d$ 
because of the affine Lefschetz theorem and the fact $\cd(s)=1$.
This proves the desired assertion in this case.
\medbreak

{\bf Case II-3:} 
Finally, to show the general case, we use an induction argument as in Case~I-3 
to reduce either to Case~II-2 or to the geometric case in Example~\ref{exHK1}  
over the base $\Spec(k)$, where $k$ is the finite residue field of $R$.
\medbreak

This completes the proof of Theorem \ref{thm.Lcondition}
\end{proof}


\section{Pullback Map (First Construction)} \label{pullback}

\bigskip

\noindent 
Let the notation be as in Section~\ref{lefschetz}.
Let $H$ be a homology theory on the category $\cCS$ leveled above $e$. 
We now introduce a condition for $H$ which plays a crucial role in
the proof of the main theorem.
\medbreak

We say that a morphism $f:Y\to X$ in $\cCS$ is embeddable if it factors as
\[
Y \stackrel{i}{\lr} P_X  \stackrel{\pi_X}{\lr} X
\] 
where $i$ is an immersion and $\pi_X$ is the base change to $X$ of a smooth morphism $\pi:P\to S$.
\medbreak

We now consider the condition:
\medbreak

({\bf PB}): For any embeddable  morphism $f:X'\to X$ between irreducible regular schemes in $\cCS$ with $\dim_S(X')=\dim_S(X)$
there exist functorial pullbacks 
\[
f^*:H_a(X) \to H_a(X')  \;\;\;  \text{ and } \;\;\; f^*:KH_a(X) \to KH_a(X')
\]
extending the given pullbacks for open immersions and
which satisfy the following properties:

\renewcommand{\labelenumi}{(\roman{enumi})}
\begin{enumerate}
\item The diagram
\[
\xymatrix{
H_{a+e} (X) \ar[r]^{\epsilon^a_{X'}} \ar[d]_{f^*} & KH_a(X) \ar[d]^{f^*} \\
H_{a+e} (X')  \ar[r]_{\epsilon^a_X} & KH_a(X') 
}
\]
commutes.
\item If $f$ is proper and dominant, the composite map
\[
KH_a(X) \rmapo{f^*} KH_a(X')  \rmapo{f_*} KH_a(X)
\]
is multiplication by $\deg(X'/X)$.
\end{enumerate}
\medbreak

The main result of this section is the following:

\begin{theo}\label{PBthm}
All homology theories in the examples in Section~\ref{homology} satisfy 
{\rm (}{\bf PB}{\rm )}.
\end{theo}
The theorem will be shown here only in case each element of the coefficient group $\Lam$ is killed by an integer invertible on the base scheme $S$. The general case requires completely revisiting Rost's pullback construction \cite{R} in the absence of purity and will be explained elsewhere.

The above theorem is deduced from the following more general statement:

\begin{theo}\label{PBTheo}
Assume that $H$ extends to a homology theory with duality, see Definition~\ref{def.Hdual}.
Then $H$ satisfies {\rm (}{\bf PB}{\rm )}.
\end{theo}

It is an easy routine to verify that the homology theories in the examples 
in Section~\ref{homology} satisfy the assumption of Theorem \ref{PBTheo} under the above restriction on the characteristic
(see Examples \ref{PBexam1}, \ref{PBexam2} and \ref{PBexam3}).
An alternative shorter proof of Theorem~\ref{PBthm} for the homology $H^\et$ from Example~\ref{exHK1} is given in the next section.
\medbreak

The following definition of homology theory with duality is a combination of the notion of a twisted Poincar\'e duality theory, see 
\cite{BO} and \cite{J1}, and of Rost's cycle modules~\cite{R}. As a general remark we mention that in \cite{R} it is assumed that the base is a field. Nevertheless the results that
we tranfer from~\cite{R} hold over a Dedekind ring too.

\begin{defi}\label{def.HTD}
A cohomology theory with supports is given by abelian groups $H^a_Y(X,n)$ 
for any closed immersion $Y\hookrightarrow X$ of schemes in $\cCS$ and for any $a,n\in \Z$,
which satisfies the following conditions. For $Y=X$, $H^a_Y(X,n)$ is simply denoted by
$H^a(X,n)$. 

\renewcommand{\labelenumi}{(\arabic{enumi})}
\begin{enumerate}

\item The groups $H^a_Y(X,n)$ are contravariant functorial for Cartesian squares
\[
\xymatrix{
Y' \ar@{^{(}->}[r] \ar[d] & X' \ar[d]^f \\
Y \ar@{^{(}->}[r]  & X  ,
}
\]
i.e.~there exists a pullback $f^*: H^a_Y(X,n) \to H^a_{Y'}(X',n)$.

\item For a closed immersion $Y\hr X$ there is a cup-product
\[
\Gamma (\OO_X^\times ) \otimes H^a_Y(X,n) \lr H^{a+1}_Y(X,n+1)
\]
compatible with pullback.
For $t\in \Gamma(\OO_X^\times )$ we denote the homomorphism $$t\cup - :  H^a_Y(X,n) \to H^{a+1}_Y(X,n+1) $$
by the symbol $\{t\} $.

\item \label{pullback.locseq} For closed immersions $Z\hookrightarrow Y\hookrightarrow X$ there is a long exact sequence 
\[
\cdots \lr H^a_Z(X,n) \lr H^a_Y(X,n) \lr H^a_{Y-Z}(X-Z,n) \stackrel{\partial}{\lr} H^{a+1}_Z(X,n)\lr \cdots
\]
compatible with pullbacks. Note that this implies in particular
\[
H^a_Y(X,n) \simeq H^a_{Y_{red}}(X,n)
\]
for a closed immersion $Y\hookrightarrow X$ and the reduced part $Y_{red}$ of $Y$.

\item Consider a Cartesian diagram
\[
\xymatrix{
Y' \ar@{^{(}->}[r] \ar[d] & X' \ar[d]^f \\
Y \ar@{^{(}->}[r]  & X  
}
\]
such that $f$ is finite and flat and $X,X'$ are regular. Then there is a pushforward map
\[
f_*:H^a_{Y'}(X',n)\to  H^a_{Y}(X,n)
\]
compatible with the localization sequence from~\eqref{pullback.locseq} and such that
$f_* \circ f^* $ equals multiplication by $\deg(X'/X)$ if $X$ is connected.

\item  For  closed immersions $Z\hr Y\hr X$  and for $t\in \Gamma (\OO_X^\times )$ the diagram
\[
\xymatrix{
 H^a_{Y-Z}(X-Z,n) \ar[r]^{\partial} \ar[d]_{-\{t\}} & \ar[d]^{\{t\}} H^{a+1}_Z(X,n) \ar[r] & H^a_Y(X,n) \ar[d]^{\{t\}} \\
  H^{a+1}_{Y-Z}(X-Z,n+1) \ar[r]_{\partial}  &  H^{a+2}_Z(X,n+1)    \ar[r]  &   H^{a+1}_Y(X,n+1) 
}
\]
commutes.

\item If $Z\hr Y\hr X$ are  closed immersions of schemes in $\cCS$, such that $Y$
has pure codimension $c$ in $X$ and $X$ is regular, then there is a Gysin homomorphism
\[
P_{X,Y} :  H^a_Z(Y,n) \lr   H^{a+2c}_Z(X,n+c) 
\]
transitive in an obvious sense and such that for closed immersions $Z\hr Z'\hr Y$ the diagram
\[
\xymatrix{
  H^a_Z(Y,n)   \ar[r] \ar[d]_{P_{X,Y}}  &   H^a_{Z'}(Y,n) \ar[r]  \ar[d]_{P_{X,Y}} &   H^a_{Z'-Z}(Y-Z,n)  \ar[r]^{\partial} \ar[d]_{P_{X,Y}}  &    H^{a+1}_Z(Y,n)  \ar[d]_{P_{X,Y}}  \\
H^{a+2c}_Z(X,n+c)  \ar[r]   &  H^{a+2c}_{Z'}(X,n+c)    \ar[r]    & H^{a+2c}_{Z'-Z}(X-Z,n+c)  \ar[r]^\partial  & H^{a+2c+1}_Z(X,n+c)    
}
\]
commutes.

\item If $X$ and $Y$ are regular, the Gysin morphism $P_{X,Y}$ is an isomorphism, called purity isomorphism.

\item For closed immersions  $Z\hr Y\hr X$ as in (6) and for $t\in \Gamma (\OO_X^\times )$ the diagram
\[
\xymatrix{
 H^*_Z(Y,*) \ar[d]^{\{ \bar t \}}  \ar[r]^{P_{X,Y}} &  H^*_Z(X,*)  \ar[d]_{\{t\}}  \\
  H^*_Z(Y,*)  \ar[r]_{P_{X,Y}}    &  H^*_Z(X,*)       }
\]
commutes. Here $\bar t$ is the restriction of $t$ to $Y$.

\item
Consider Cartesian squares 
\begin{equation}\label{pull.carsq}
\xymatrix{
Z' \ar@{^{(}->}[r] \ar[d] &  Y' \ar@{^{(}->}[r] \ar[d]_g & X' \ar[d]^f \\
Z \ar@{^{(}->}[r] & Y \ar@{^{(}->}[r]  & X  .
}
\end{equation}
with $f$ flat and $X,X',Y,Y'_{red}$ regular. Let $\eta^i_Y$ for $1\leq i\leq r$ be the maximal points of $Y'$  and set $m^i = \mathrm{length}(\OO_{Y',\eta^i_Y}) $. 
Then the diagram
\[
\xymatrix{
 H^*_Z(Y,*) \ar[d]^{ \sum m^i\, (g^{i})^* } \ar[r]^{P_{X,Y}} &  H^{*}_{Z}(X,*)  \ar[d]_{f^*} \\ 
  H^*_{Z'}(Y',*)  \ar[r]_{P_{X',Y'}}   &  H^{*}_{Z'}(X',*) 
}
\]
commutes. Here $g^i$ means the induced morphism $\overline{\{ \eta^i \} } \to Y$.

\item Consider  Cartesian squares as in \eqref{pull.carsq}
with $f$ finite and flat. Then
the diagram
\[
\xymatrix{
H^{*}_{Z'}(Y',*)\ar[d]_{f_*} \ar[r]^{P_{X',Y'}}  &  H^{*}_{Z'}(X',*) \ar[d]_{f_*} \\
  H^{*}_Z(Y,*)   \ar[r]^{P_{X,Y}}    & H^{*}_{Z}(X',*)
}
\] 
commutes.

\item For an equidimensional regular  scheme $X$ and $t\in \Gamma (\OO_X )$ such that $Y:=V(t)$ is regular,
the diagram 
\[
\xymatrix{
H^a(Y,n) \ar[rr]_{P_{X,Y}}  & &  H^{a+2}_Y(X,n+1)  \\
 \ar[r]H^a(X,n) \ar[u]_{i^*} & \ar[r]^/-0.8em/{\{t\}} H^a(U,n) & \ar[u]_{\partial} H^{a+1}(U,n+1)  
}
\]
commutes. Here $U=X-Y$ and $i:Y\hr X$ is the closed immersion.

\end{enumerate}
\end{defi}

\medskip

\begin{defi}\label{def.Hdual}
We say that a homology theory $H$ extends to a homology theory with duality if there exists a cohomology theory with supports 
$(X,Y)\mapsto H_Y^*(X,*)$ in the sense of Definition~\ref{def.HTD}
and an integer $n_0$ 
such that for a regular equidimensional $X\in \cCS$ of dimension $d$  there is an isomorphism 
$$D_X : H_a(X) \stackrel{\sim}{\lr} H^{2d-a}(X,d+n_0)$$
with the following properties:
\renewcommand{\labelenumi}{(\arabic{enumi})}
\begin{enumerate}
\item For a closed immersion of regular schemes $Y\hr X$ and for $U=X-Y$ the diagram
\[
\xymatrix{
  H_*(U)  \ar[r]^\partial \ar[d]^{D_U}  & H_*(Y)  \ar[r] \ar[d]_{D_Y} & H_*(X)  \ar[d]_{D_X}   \\
H^*(U,*) \ar[r]^\partial &  H^*(Y,*) \ar[r]  &  H^*(X,*) 
}
\]
commutes. Here the lower row is a combination of purity in \ref{def.HTD}(7) and 
the localization sequence in \ref{def.HTD}(3).

\item For an open immersion $U\to X$ of regular schemes in $\cCS$ the diagram
\[
\xymatrix{
H_*(X) \ar[r] \ar[d]^{D_X} &  H_*(U)  \ar[d]_{D_U} \\
H^*(X,*)  \ar[r]  &  H^*(U,*)
}
\]
commutes.

\item If $f:X' \to X$ is finite and flat and $X,X'$ are regular the diagram
\[
\xymatrix{
H_*(X')  \ar[d]^{f_*}  \ar[r]_{D_{X'}} &  H^*(X',*) \ar[d]_{f_*} \\
 H_*(X)  \ar[r]_{D_{X}} &  H^*(X,*) 
}
\]
commutes
\end{enumerate}
\end{defi}

\bigskip

The following examples show that the homology theories in the examples in Section~\ref{homology} 
extend to homology theories with duality under our assumption on characteristic. 
\medbreak

\begin{exam}\label{PBexam1}
Let $S=\Spec(k)$ be the spectrum of a perfect field and
fix a $G_k$-module $\Lambda$ whose elements are of fintie order prime to $\chara(k)$.
We define \'etale homology of $\pi : X\to S$ to be
\[
H^\et_a (X) = H^{-a}(X_\et , R\, \pi^! \Lambda  )
\]
and \'etale cohomology with support in $Y\hr X$ to be
\[
H_Y^a (X,n) = H^a_Y(X_\et , \Lambda (n)) .
\]
This defines a cohomology theory with support in the sense of Definition~\ref{def.HTD}  and there is an isomorphism
$$H_a(X)\cong H^{2d-a}(X,d)$$
for regular equidimensional $X$ of dimension $d$ by \cite[XVIII Theorem 3.2.5]{SGA4} .
\end{exam}
\medbreak

\begin{exam}\label{PBexam2}
Let $S=Spec(k)$ be the spectrum of the perfection of a finitely generated field and
fix a finite $G_k$-module $\Lambda$ of order prime to $\chara(k)$. 
Let $d_k$ be the Kronecker dimension of $k$ and $d$ the dimension of $X$.
We define homology of $\pi:X\to S$ to be
\[
H^D_a(X) = \Hom (H^a_c (X_\et ,\Lambda^\vee ),\Q/\Z )
\]
and cohomology with support in $i:Y\hr X$ to be
\[
H_Y^a(X,n) =  H^{a+2d_k+1}_c( k, R \, \pi^Y_*\, R\, i^!\, \Lambda (n+d_k) ) 
\]
where $\pi^Y: Y \to \Spec(k)$ is the canonical morphism.
For the definition of Galois cohomology of $k$ with compact support and for Artin-Verdier duality see Section~\ref{duality}. If $\Lambda$ is a torsion $G_k$-module 
whose elements are of order prime to $\chara(k)$,
we define the corresponding homology and cohomology groups as the direct limit over all finite submodules.
This defines a cohomology theory with support and we get that $H^D$ is a homology theory with duality (see Corollary~\ref{coro.compact0}).
\end{exam}
\medbreak

\begin{exam}\label{PBexam3}
Let $S=\Spec(R)$ be the spectrum of a henselian discrete valuation ring.
Denote by $s$ (resp.~$\eta$) the closed (resp.~generic) point of $S$. 
Fix a unramified $G_\eta$-module $\Lambda$ whose elements are of finite order prime to $\chara(k(s))$. Define homology of $\pi :X\to S$ to be
\[
H^\et_a(X) = H^{2-a}(X_\et , R\, \pi^! \Lambda (1)) 
\]
and cohomology with support in $Y\hr X$ to be
\[
H_Y^a (X,n ) = H^{a}_Y(X_\et ,\Lambda (n) ) .
\]
This defines a cohomology theory with support (see \cite[Lemma 1.8]{SS}).
\end{exam}
\bigskip

Assume we are given a homology theory $H$. 
Recall that the niveau spectral sequence for $X\in \cCS$ reads
\[
E^1_{a,b}(X) = \bigoplus_{x\in X_{(a)}} H_{a+b}(x)  \Longrightarrow H_{a+b}(X) .
\]
It has a pushforward morphism induced by a proper morphism of schemes~\cite{BO}.
On the other hand the coniveau spectral sequence for a cohomology theory with support reads
\[
E_1^{a,b}(Y\hr X,n) = \bigoplus_{x\in X^{(a)} \cap Y} H_x^{a+b}(X,n) \Longrightarrow H^{a+b}_Y(X,n) 
\]
for a closed immersion $Y\hr X$ and for $n\in \Z$.  
If a homology theory $H$ extends to a homology theory with duality in the sense of
Definition~\ref{def.Hdual} then the two spectral sequences are canonically isomorphic up to indexing. For example we have isomorphisms for regular $X=Y$ of dimension $d$:
\[
E_1^{a,b}(X,d+n_0) \cong E^1_{d-a,d-b}(X).
\]

In order to construct a pullback for Kato homology we will work with cohomology theories with support.
Our next definition makes the term `morphism of coniveau spectral sequences' precise.

\begin{defi}\label{def.MCSS}
For $X,Y \in \cCS$ a morphism of coniveau spectral sequences \[ \xymatrix{f:X \ar@{~>}[r]& Y} \] (of degree $(a_0,b_0,m_0)$ and level $t\ge 1$) is given by the following data:
\begin{enumerate}
\item For $r\ge t$ there are homomorphisms 
\[
f_r : E_r^{a,b} (X,n) \lr E_r^{a+a_0,b+b_0} (Y,n+m_0)
\]
such that the following diagram is commutative
\[
\xymatrix{
E_r^{a,b} (X,n) \ar[rr]^{(-1)^{a_0+b_0} d_r} \ar[d]_{f_r} & & E_r^{a+r,b-r+1} (X,n)  \ar[d]^{f_r} \\
E_r^{a+a_0,b+b_0} (Y,n+m_0) \ar[rr]_{d_r}  & & E_r^{a+r+a_0,b-r+1+b_0} (Y,n+m_0)  
}
\]
and such that $f_{r+1}$ and $f_r$ are compatible with respect to the diagram.

\item There are homomorphisms 
\[
f_{tot} : H^{a+b}(X,n)  \lr  H^{a+b+a_0+b_0}(Y,n+m_0 ) 
\]
compatible with the limit homomorphisms $$f_\infty : E_\infty^{a,b} (X,n) \lr E_\infty^{a+a_0,b+b_0} (Y,n+m_0) .$$
\end{enumerate}
\end{defi}
\medbreak

In the following we give some examples of morphisms of coniveau spectral sequences. 
\medbreak

\noindent {\bf Flat pullback:} For a flat morphism $f:Y\to X$ of schemes in $\cCS$ it is shown in \cite{BO} that there arises from \ref{def.HTD}(1) a morphism of
coniveau spectral sequences
\[
 \xymatrix{f^*:X \ar@{~>}[r]& Y} 
\]
of degree $(0,0,0)$ and level $1$.

\medskip

\noindent {\bf Cup-Product:} Let $X$ be in $\cCS$.  By taking  cup-product  with an element  $t\in \Gamma (\OO_X^\times )$ (cf. \ref{def.HTD}(2)), we get a morphism
of coniveau spectral sequences
\[
\xymatrix{ \{t\} :  X \ar@{~>}[r]& X  }
\]
of degree $(0,1,1)$ and level $1$.

\medskip

\noindent {\bf Pushforward by closed immersion:}  
For a closed immersion $i:Y\hr X$ of codimension $c$ of regular schemes in $\cCS$, 
we get a pushforward morphism of coniveau spectral sequences
\[
\xymatrix{
i_* :  Y \ar@{~>}[r] & X   }
\]
of degree $(c,c,c)$ and level $1$ which is given as the composition 
\[ 
E_r^{a-c,b-c}(Y,n-c) \stackrel{\sim}{\lr} E_r^{a,b}(Y\hr X,n) \lr E_r^{a,b}(X,n).
\]
Here the former morphism is induced by purity in \ref{def.HTD}(7).

\medskip

\noindent {\bf Residue:}  
For a closed immersion $Y\hr X$ of codimension $c$ of regular schemes in $\cCS$ with $U=X-Y$, we get a residue morphism of coniveau spectral sequences
\[
\xymatrix{
\partial : U  \ar@{~>}[r] & Y   }
\]
of degree $(1-c,-c,-c)$ and level $1$ which is given as the composition 
\[ 
E_r^{a,b}(U,n) \lr E_r^{a+1,b}(Y\hr X,n) \stackrel{\sim}{\lr} E_r^{a+1-c,b-c}(Y,n-c) .
\]
Here the former morphism arises from \ref{def.HTD}(3) and the latter isomorphism 
is induced by purity in \ref{def.HTD}(7). For $r=2$ we have a long exact sequence
\begin{equation}\label{eq.locseq}
\cdots \rmapo{\partial} E_2^{a-c,b-c}(Y,n) \rmapo{i_*} E_2^{a,b}(X,n) 
\rmapo{j^*} E_2^{a,b}(U,n) \rmapo{\partial}  E_2^{a+1-c,b-c}(Y,n) \to \cdots\\
\end{equation}
where $j:U\hookrightarrow X$ is the open immersion.
\medskip

\noindent {\bf Finite flat Pushforward:}
For a finite flat morphism $f:X'\to X$ between regular schemes in $\cCS$,
there arises from \ref{def.HTD}(4) a morphism of coniveau spectral sequences
\[
 \xymatrix{f_*:X' \ar@{~>}[r]& X} 
\]
of degree $(0,0,0)$ and level $1$.

\medskip

The following proposition on homotopy invariance is crucial. 

\begin{prop}
For regular $X\in \cCS$  and for a vector bundle $f:V\to X$ the flat  pullback along $f$ induces  an isomorphism of coniveau spectral sequences of degree $(0,0,0)$ and level $2$
\[
\xymatrix{f^* : X \ar@{~>}[r] & V \;.  }
\]
In particular $f^*:H^*(X,*) \to H^*(V,*)$ is an isomorphism.
\end{prop}
\begin{proof}
It is sufficient to show that $f$ induces an isomorphism on $E_2^{*,*}$-level.
First we treat the special case where $V\to X$ is a trivial vector bundle, namely
$V=X\times \mathbb{A}^n \to X$ is the natural projection. 
By induction on $n$ we are reduced to the case $V=X\times \mathbb{A}^1$.
This can be shown by transferring the argument of \cite[Section 9]{R} to our setting.
For the function field $E$ of an integral scheme $X\in \cCS$, write
\[
H^*(E,*)=\indlim {U\subset X} H^*(U,*),
\]
where $U$ ranges over all open subschemes of $X$. Then the functor
$E\mapsto H^*(E,*)$ from function fields of integral schemes in $\cCS$ 
to abelian groups satisfies properties similar to Rost's axioms for cycle modules.  
Now an inverse to $f_2^*:E_2^{*,*}(X) \to E_2^{*,*}(V)$ is given by
\[
\xymatrix{
X\times \mathbb{A}^1 \ar@{~>}[r]  & X\times (\mathbb{A}^1- \{ 0\}) \ar@{~>}[r]^{\{ -1/t\} }  &  X\times (\mathbb{A}^1- \{ 0\} )
\ar@{~>}[r]^/1.4em/{\partial}  & X .
}
\]
Here $\partial$ is the residue morphism corresponding to the closed immersion $X\times \{\infty \} \hr  X\times (\mathbb{P}^1- \{ 0\} )$. 

We now reduce the general case to the special case.
We proceed by induction on $\dim(X)$. First we assume that there exists
a regular closed subscheme $Y\subset X$ such that $V\to X$ is a trivial vector bundle
over $U=X-Y$. We have the commutative diagram
\[
\xymatrix{
\cdots \ar[r]^/-1.1em/{\partial}   & E_2^{a-c,b-c}(Y,n) \ar[d]^{f^*}  \ar[r]^{i_*} &  E_2^{a,b}(X,n)   \ar[r]^{j^*} \ar[d]^{f^*} & E_2^{a,b}(U,n) \ar[r]^/0.5em/{\partial}  \ar[d]^{f^*} & \cdots \\
\cdots  \ar[r]^/-1.8em/{\partial}   &    E_2^{a-c,b-c}(f^{-1}(Y),n) \ar[r]^/0.4em/{i_*}  &   E_2^{a,b}(f^{-1}(X),n)  \ar[r]^{j^*}  & E_2^{a,b}(f^{-1}(U),n) 
\ar[r]^/1.4em/{\partial}  &  \cdots
}
\]
where the horizontal long exact sequences come from \eqref{eq.locseq}.
By induction and by what we have shown, this proves that
$f_2^*:E_2^{*,*}(X) \to E_2^{*,*}(V)$ is an isomorphism.

In general there is a sequence of closed subschemes
\[
X_0\subset X_1\subset \cdots X_r=X
\]
such that $X_i-X_{i-1}$ for $1\leq i\leq r$ are regular and that
$V\to X$ is a trivial vector bundle over $X-X_r$.
This reduces the proof to the previous case and the proof is complete.
\end{proof}
\medbreak

Using homotopy invariance we can now define the pullback morphism of coniveau spectral sequences along closed immersions of regular schemes.
For such a closed immersion $i:Y \hr X$ let $D(X,Y)$ be the blow-up of $X\times \mathbb{A}^1$ in $Y\times \{0 \}$  minus the blow-up of $X$ in $Y$.
Then $D(X,Y)$ is flat over $\mathbb{A}^1$, and the fiber over a point 
$p\in \mathbb{A}^1 - \{0\}$ (resp. $0\in \mathbb A^1$) is isomorphic to 
$X\times \{ p \}$ (resp. the normal cone $N_Y X$ of $Y$ in $X$), cf.~\cite[Sec.\ 10]{R}.

\medbreak

\noindent {\bf Pullback along immersions:}
Let $i:Y\to X$ be a closed immersion of regular schemes in $\cCS$. Using the normal cone we define the morphism of coniveau spectral sequences of degree $(0,0,0)$ and level $1$
\[
\xymatrix{ J(i) : X \ar@{~>}[r] &  N_Y X . } 
\]
as the composition 
\[
\xymatrix{ 
X \ar@{~>}[r]^/-1.5em/{\pi^*}  & X\times (\AAA^1 - \{0\} )  \ar@{~>}[r]^{\{t\}}  
& X\times (\AAA^1 - \{0\} ) \ar@{~>}[r]^/1em/{\partial} & N_Y X  .}
\] 
Here $\pi:X\times (\AAA^1 - \{0\} )\to X$ is the projection and 
$t$ is the standard coordinate of $\AAA^1$ and $\partial$ is the residue morphism corresponding to the closed immersion $N_Y X \hr D(X,Y)$.
Composing $J(i)$ with the inverse of the flat pullback along $N_Y X\to Y$ we get the pullback morphism along $i$ which is of degree $(0,0,0) $ and level $2$
\[
\xymatrix{ i^* : X \ar@{~>}[r] &  Y  .} 
\]
\medbreak

The proof of the next proposition is standard
(see ~\cite{Deg} for the argument on the $E_2^{*,*}$-level).

\begin{prop}\label{compimpull}
 The pullback along closed immersions satisfies:

\renewcommand{\labelenumi}{(\roman{enumi})}
\begin{enumerate}
\item The homomorphism $i^*_{tot}:H^*(X,*) \to H^*(Y,*) $ stemming from the normal cone construction is equal to the usual pullback from Definition \ref{def.HTD}(1).
\item 
The construction of pullback along closed immersions is functorial.
\item It is compatible with cup-products.

\item For a Cartesian square
\[
\xymatrix{
Y'  \ar[d] \ar[r] &  X' \ar[d]^{f} \\
Y \ar[r]_{i} &  X
}
\]
of regular schemes 
with closed immersion $i$ and flat $f$ the diagram
\[
\xymatrix{
Y'   &  X' \ar@{~>}[l]  \\
Y \ar@{~>}[u] &  X \ar@{~>}[l]^{i^*}  \ar@{~>}[u]_{f^*}
}
\]
commutes.

\item For two morphisms $Z\stackrel{i}{\to} Y \stackrel{p}{\to} X$ of regular schemes with $i$ a closed immersion, $p$ flat and $p\circ i$ flat,
we have
\[
(p\circ i)^* = i^* \circ p^* .
\]

\item For two morphisms $Z\stackrel{i}{\to} Y \stackrel{p}{\to} X$ of regular schemes with $i$ a closed immersion, $p$ smooth and $p\circ i$ a closed immersion,
we have
\[
(p\circ i)^* = i^* \circ p^* .
\]
\end{enumerate}
\end{prop}
\medbreak

It follows from Proposition~\ref{compimpull}~(iv) that one gets a pullback map along locally closed immersions of regular schemes with the same properties.
Using Proposition~\ref{compimpull}~(vi) and standard arguments, see~\cite{F} Section~6.6, one can now define a pullback along arbitrary embeddable morphisms between regular schemes.

\smallskip

\noindent {\bf General Pullback:} Let $f: Y\to X$ be an embeddable morphism of regular schemes  in $\cCS$, which factors as
\[
Y \stackrel{i}{\lr} P \stackrel{\pi}{\lr} X
\]
where $\pi$ is smooth and $i$ is an immersion. Define
\[
f^*= i^* \circ \pi^* .
\]

\begin{prop}\label{compgenpull}
The following holds for the general pullback along $f$:
\renewcommand{\labelenumi}{(\roman{enumi})}
\begin{enumerate}
\item The map $f_{tot}^*:H^*(X,*)  \to H^*(Y,*)$ is equal to the usual pullback map of cohomology.
\item The map $f^*$ does not depend on the factorization of $f$.
\item The general pullback construction is functorial. 
\item If $f$ is flat the general pullback coincides with the flat pullback.
\end{enumerate}
\end{prop}
\bigskip

Now we can prove Theorem \ref{PBTheo}.
Passing from the coniveau spectral sequence to the niveau spectral sequence as discussed above gives us  
a pullback map of niveau spectral sequences for a morphism of regular schemes of the same dimension. In particular
we have produced the pullback maps in condition 
{\rm (}{\bf PB}{\rm )}.
The only thing one has to check is part (ii) of {\rm (}{\bf PB}{\rm )}. For this we need a detailed understanding of the
pullback construction on $E_1^{*,*}$-level as is given in the work of Rost~\cite{R} in terms of his Chow groups with coefficients.
Fix an integer $a\geq 0$.
Recall that for $X\in \cC$, the term in degree $a$ of the Kato complex $\KC X$ is
$$
E^1_{a,e}(X)=\sumd X a H_{a+e}(x) .
$$
Now let $f:X'\to X$ be as in {\rm (}{\bf PB}{\rm )}.
Choose a non-empty open subscheme $V\subset X$ such that $f$ is finite and flat over 
$V$. Putting $Z=X-V$, we have
\begin{equation}\label{eq.pullback}
E^1_{a,e}(X)=E^1_{a,e}(V)\oplus E^1_{a,e}(Z). 
\end{equation}
We put
$$
\Sigma=\ker\big(E^1_{a,e}(V)\hookrightarrow E^1_{a,e}(X) \rmapo{d^1} E^1_{a-1,e}(X)\big)
$$
and denote by $\im^V_*$ the image of $\Sigma$ in $\KH a X=E^2_{a,e}(X)$. 
The following two lemmas imply part (ii) in {\rm (}{\bf PB}{\rm )} and therefore finish the proof the the theorem.

\begin{lem}\label{onim}
The restriction of $f_*\circ f^*:\KH a X\to \KH a X$ to $\im^V_*$, i.e.
\[
f_*\circ f^* |_{\im^V_*} : \im^V_* \lr \KH a X ,
\]
is equal to the inclusion $\im^V_*\hr \KH a X$ multiplied by $\deg(X'/X)$.
\end{lem}

\begin{lem}\label{imall}
The inclusion  $\im^V_* \hr \KH a X$ is an isomorphism.
\end{lem}

\begin{proof}[Proof of Lemma \ref{onim}]

Consider $\alpha \in \Sigma\subset E^1_{a,e}(V)$ and denote its image in $E^1_{a,e}(X)$ by $\alpha_X$ (cf. \eqref{eq.pullback}). 
Choose a factorization $f=i \circ p$ with a closed immersion $i$ and
a smooth morphism $p$ and choose a coordination of the normal bundle of $i$ as in \cite[Sec.\ 9]{R}. This choice induces a pullback map
$f^*:E^1_{a,e}(X)\to E^1_{a,e}(X')$ as in \cite[Sec.\ 12]{R} which satisfies
the following properties:
\begin{itemize}
\item
the restriction of $f^*$ to
$\ker(d^1:E^1_{a,e}(X) \to E^1_{a-1,e}(X))$ induces 
$f^*$ on $\KH a X=E^2_{a,e}(X)$.
\item
For an open immersion $j:U\hookrightarrow X$,
we have the commutative diagram:
\begin{equation}\label{eq2.pullback}
\begin{CD}
E^1_{a,e}(X)@>{f^*}>> E^1_{a,e}(X')\\
@VV{j^*}V @VV{(j')^*}V \\
E^1_{a,e}(U)@>{f_U^*}>> E^1_{a,e}(U')\\
\end{CD}
\end{equation}
where the vertical arrows are the pullbacks by the open immersions
and $j':U'=f^{-1}(U) \hookrightarrow X'$ and 
$f_U:U'\to U$ is the base change of $f$ via $j$ (see the sentence above 
\cite{R}, Lemma 11.1).
\item
In the above diagram, if $f_U$ is flat, then
$f_U^*$ is the naive pullback via the flat map $U'\to U$
(see below Definition \ref{def.MCSS}, and \cite[(3.5)] {R} and \cite[Lemma 12.2]{R}).
\end{itemize}
We put 
$$\alpha_{X'}=f^*(\alpha_X)\in E^1_{a,e}(X').$$
Let $V'=f^{=1}(V)$ and $Z'=X'-V'=f^{-1}(Z)$ with $Z=X-V$. Let
$$
\alpha_{X'}=\alpha_1+\alpha_2,\quad 
\alpha_1\in E^1_{a,e}(V'),\; \alpha_2\in E^1_{a,e}(Z')
$$
be the decomposition with respect to 
\[
E^1_{a,e}(X')=E^1_{a,e}(V') \oplus E^1_{a,e}(Z').
\]
By the above fact, $\alpha_1=(j')^*(\alpha_{X'})$ 
is the naive pullback of $\alpha\in E^1_{a,e}(V)$ via the flat map $V'\to V$. 
Therefore we have $f_*(\alpha_1)=\deg(X'/X)\, \alpha_X$.
So in order to finish the proof of the lemma it suffices to show that $\beta:=f_*(\alpha_2)=0\in E^1_{a,e}(X)$. 
We have the commutative diagram:
\[
\begin{CD}
E^1_{a,e}(Z')@>{i'_*}>> E^1_{a,e}(X')\\
@VV{(f_Z)_*}V @VV{f_*}V \\
E^1_{a,e}(Z)@>{i_*}>> E^1_{a,e}(X)\\
\end{CD}
\]
where $i:Z\hookrightarrow X$ and $i':Z'\hookrightarrow X'$ are the closed immersions
and $f_Z:Z'\to Z$ is the base change of $f$ via $i$.
This shows the support of $\beta$ is contained in $Z$, namely $\beta\in E^1_{a,e}(Z)$.
 Let $i:Y\hr X$ be the support of $\alpha_X$. Since $\alpha_X$ is the image of
$\alpha\in E^1_{a,e}(V)$, $Y$ is the union of closed
integral subschemes $Y_i$ of $\dim_S(Y_i)=a$ such that none of $Y_i$ is
contained in $Z=X-V$, which implies $\dim_S(Y\cap Z)<a$.
Applying the diagram \eqref{eq2.pullback} to the case $U=X-Y$,
we see $\alpha_{X'}\in E^1_{a,e}(Y')$ with $Y'=f^{-1}(Y)$.
By the same argument as above, this implies the support of $\beta$ is contained in 
$Y$ and hence $\beta\in E^1_{a,e}(Z\cap Y)$.
Now the desired assertion follows from the fact $E^1_{a,e}(Z\cap Y)=0$ as $\dim_S(Z\cap Y)<a$. 
\end{proof}
\medbreak

\begin{proof}[Proof of Lemma \ref{imall}]

Assume for simplicity of notation that our homology theory $H$ is leveled above $e=0$.
It suffices to show that the composite map
\[
E^1_{a+1,0}(X) \rmapo{d^1} E^1_{a,0}(X) \to E^1_{a,0}(Z) 
\]
is surjective for $a\in \Z$, where the second map is the projection with respect to
\eqref{eq.pullback}. For a point $z\in Z_{(a)}$ we can choose $z'\in V_{(a+1)}$ 
by \cite[Lemma 7.2]{SS} such that $z$ lies on the regular locus of 
$Y:=\overline{\{ z' \}}$.
Let $R$ be the finite set of points of dimension $a$ on $Y\cap Z$ which are different from $z$. We have an exact sequence arising as the limit of localization sequences for
the homology theory $H$
\[
H_{a+1}(z') \rmapo{\partial} H_a(z) \oplus H_a(R') \lr H_a(A),
\]
where $A$ is the normalization of the semi-local one-dimensional ring $\OO_{Y,\{z\}\cup R}$ and $R'$ is the set of points of $\Spec(A)$ lying over $R$. 
By the assumption that $H$ is leveled above $e=0$ we have $H_a(A)=0$ and hence
$\partial$ is surjective. This implies the desired surjectivity.
\end{proof}


\section{Pullback Map (Second Construction)}\label{pullbacksnd}

\bigskip

\noindent
In the previous section we gave a complete proof of the fact, Theorem~\ref{PBthm}, that the homology theories defined in the examples in Section~\ref{homology}
satisfy condition $\rm{(}\bf{PB}\rm{)}$ under our restriction on characteristic. This proof relies on the Fulton-Rost version of intersection theory. In this section
we give an alternative proof of Theorem~\ref{PBthm} for the \'etale homology theory $H^\et$ from Example~\ref{exHK1} by using the fact that the Gersten
conjecture holds over fields. This allows us to remove the restriction on characteristic at least for $H^\et$.
\medbreak

Let the notation be as in Example~\ref{exHK1}. Assume $S=\Spec(k)$, where $k$ is a finite field.
Let $X$ be smooth over $S=\Spec(k)$ of pure dimension $d$. 
By Lemma~\ref{lemHK1smooth} we have an isomorphism
$$
H_a(X)=\Het_a(X,\Lam) = H^{2d-a}_{\et}(X,\Lam(d)). 
$$
Thus, for $f:X\to Y$ as in $\bf{(PB)}$,
$f^*:\Het_a(Y,\Lam)\to \Het_a(X,\Lam)$ is defined by 
$$
f^* : H^{2d-a}_{\et}(Y,\Lam(d))\to H^{2d-a}_{\et}(X,\Lam(d)),
\quad (d=\dim(X)=\dim(Y))
$$
the pullback map for \'etale cohomology.
We have the spectral sequence
\begin{equation}\label{coniveau.sp.seq}
E_2^{a,b}=H^a_{Zar}(X,\cH^b_X)\Rightarrow H^{a+b}_{\et}(X,\Lam(d))
\end{equation}
where
$\cH^b_X=R^b\rho_* \Lam(d)$ with $\rho:X_{\et}\to X_{Zar}$, 
the natural morphism of sites. By the Bloch-Ogus version of Gersten's conjecture (see \cite{CTHK}), 
we have a flasque resolution of $\cH_X^b$:
\begin{equation}\label{flasque.res}
\cH_X^b\isom \cC^b_X\text{  in } D^b(X_{Zar}),
\end{equation}
where $D^b(X_{Zar})$ is the derived category of bounded 
complexes of Zariski sheaves on $Y$ and $\cC^b_X$ is the following (cohomological) complex of 
Zariski sheaves:
\begin{multline*}\label{Gersten.complex}
\sumcd X 0 (i_x)_* H_{\et}^{b}(x,\Lam(d))\to
\sumcd X 1 (i_x)_*H_{\et}^{b-1}(x,\Lam(d-1))\to \cdots \\
\to \sumcd X {a} (i_x)_*H_{\et}^{b-a}(x,\Lam(d-a))\to\cdots 
\end{multline*}
where the term $\sumcd X a (i_x)_*H_{\et}^{b-a}(x,\Lam(d-a))$ is placed 
in degree $a$ and $i_x:x\to X$ is the inclusion. 
In view of the description \eqref{KCfinitefield} of $\KC X$, 
this implies a natural isomorphism
\begin{equation}\label{KHGersten}
\KHet a X \isom H^{d-a}_{Zar}(X,\cH^{d+1}_X)
\end{equation}
We note $\cH^b_X=0$ for $b>d+1$ due to the injectivity of
$\cH^b_X\to (i_\eta)_* H^b(\eta,\Lam(d))$ ($\eta$ is the generic point of $X$)
and the fact $H^b(\eta,\Lam(d))=0$ for $b>d+1$ where we use the assumption that $k$ is finite.
Thus we get an edge homomoprhism for the spectral sequence
\eqref{coniveau.sp.seq}:
$$
\lambda^a_X : H^{2d+1-a}_{\et}(X,\Lam(d))\to H^{d-a}_{Zar}(X,\cH^{d+1}_X).
$$
This allows us to define $f^*:\KHet a Y \to \KHet a X$ for $f:X \to Y$ as 
in $\bf{(PB)}$ to be the composite map 
$$
f^* : H^{d-a}_{Zar}(Y,\cH^{d+1}_Y)\to H^{d-a}_{Zar}(X,f^*\cH^{d+1}_Y)\to H^{d-a}_{Zar}(X,\cH^{d+1}_X),
$$
where the second map is induced by the natural pullback map 
$f^*\cH^{d+1}_Y\to \cH^{d+1}_X$.

The first condition of $\bf{(PB)}$ is immediate from the definitions.
The second condition follows from the commutative diagram
$$
\begin{CD}
H^{2d+1-a}_{\et}(Y,\Lam(d)) @>f^*>> H^{2d+1-a}_{\et}(X,\Lam(d))\\
@VV{\lambda^a_Y}V  @VV{\lambda^a_X}V \\
H^{d-a}_{Zar}(Y,\cH^{d+1}_Y) @>f^*>> H^{d-a}_{Zar}(X,\cH^{d+1}_X)\\
\end{CD}
$$
and from the commutative diagram
$$
\begin{CD}
\Het_{a-1}(X,\Lam) @>\sim>> H^{2d+1-a}_{\et}(X,\Lam(d))\\
@VV{\edgehom a X}V @VV{\lambda^a_X}V \\
\KHet a X @>\sim>> H^{d-a}_{Zar}(X,\cH^{d+1}_X)\\
\end{CD}
$$
The commutativity of the last diagram follows from the compatibility of 
the niveau and coniveau spectral sequences with respect to the isomorphism
\eqref{KHGersten} and from compatibility of the coniveau spectral sequence and the 
spectral sequence \eqref{coniveau.sp.seq} (cf. \cite[Cor.4.4]{Pa}).

Finally we show the last condition of $\bf{(PB)}$.
Let $f:X\to Y$ be as in $\bf{(PB)}$. 
From the flasque resolution \eqref{flasque.res}, we get an isomorphism
\begin{equation}\label{flasque.res2}
\bR f_*\cH^{d+1}_X \isom f_*\cC_X^{d+1}\text{  in  } D^b(Y_{Zar}).
\end{equation}
We have the trace map
$$
tr_f: f_*\cC_X^{d+1}\to \cC_Y^{d+1}
$$ 
that induces $\Gamma(U,f_*\cC_X^{d+1})\to \Gamma(U,\cC_Y^{d+1})$
for open $U\subset Y$, which is induced by the following commutative diagram
$$
\begin{CD}
\sumcd V 0 H_{\et}^{d+1}(x,\Lam(d))@>>>
\sumcd V 1 H_{\et}^{d}(x,\Lam(d-1))@>>> \cdots @>>>
\sumcd V d H_{\et}^{1}(x,\Lam)\\
@VV{f_*}V @VV{f_*}V @. @VV{f_*}V \\
\sumcdy U 0 H_{\et}^{d+1}(y,\Lam(d))@>>>
\sumcdy U 1 H_{\et}^{d}(y,\Lam(d-1))@>>> \cdots @>>>
\sumcdy U d H_{\et}^{1}(y,\Lam)\\
\end{CD}
$$
where $V=f^{-1}(U)$ and the upper (resp. lower) complex represents $\Gamma(U,f_*\cC_X^{d+1})$ (resp. $\Gamma(U,\cC_Y^{d+1})$).
For $x\in V^{(a)}$ and $y\in U^{(a)}$, the $(x,y)$-component of $f_*$:
$$
H_{\et}^{d+1-a}(x,\Lam(d-a))\to H_{\et}^{d+1-a}(y,\Lam(d-a))
$$
is $0$ if $y\not=f(x)$ and the pushforward map for
the finite moprhism $\Spec(x)\to \Spec(y)$ if $y=f(x)$.
By \eqref{flasque.res} and \eqref{flasque.res2}, $tr_f$ induces 
$tr_f: \bR f_*\cH_X^{d+1}\to \cH_Y^{d+1}$, a map in $D^b(Y_{Zar})$, and 
$$
f_*:  H^{b}_{Zar}(X,\cH^{d+1}_X)\simeq H^{b}_{Zar}(Y,\bR f_*\cH^{d+1}_X)
\rmapo{tr_f} H^{b}_{Zar}(Y,\cH^{d+1}_Y).
$$
It is easy to check the commutativity of the following daigram:
$$
\begin{CD}
\KHet a X @>{\simeq}>> H^{d-a}_{Zar}(X,\cH^{d+1}_X) \\
@VV{f_*}V @VV{f_*}V \\
\KHet a Y @>{\simeq}>> H^{d-a}_{Zar}(Y,\cH^{d+1}_Y) \\
\end{CD}
$$
Hence we are reduced to show the following:

\begin{claim}
Let $f:X\to Y$ be as $\bf{(PB)}$. Then the composite
$$
\cH_Y^{d+1} \rmapo{f^*} \bR f_*\cH_X^{d+1}\rmapo{tr_f} \cH_Y^{d+1}
$$ 
is the multiplication by $[k(X):k(Y)]$, where the first map is 
the adjunction of the pullback map
$f^*\cH_Y^{d+1} \to \cH_X^{d+1}$.
\end{claim}

Indeed, by the Gersten conjecture, the natural map 
$\cH_Y^{d+1}\to (i_{\eta})_* H_{\et}^{d+1}(\eta,\Lam(d))$
is injective, where $\eta$ is the generic point of $Y$.
Hence it suffices to show the claim after the restriction to $\eta$, which
follows from the standard fact that the composite
$$
H^n(\eta,\Lam(d)) \rmapo{f^*} H^n(\xi,\Lam(d)) \rmapo{f_*} H^n(\eta,\Lam(d))
$$
is the multiplication by $[k(X):k(Y)]$, where $\xi$ is the generic point of $X$.


\section{Main Theorem} \label{maintheo}

\bigskip

\noindent 
In this section we state and prove the main theorem \ref{mainthm}. 
Let the notation be as in Section \ref{logpairs}.
For an integer $q\geq 0$ and a prime $\ell$, consider the following condition:

\medbreak\noindent
$\bf{(G)}_{\ell,\it{q}}$:   
For any $X\in \cSregir$ and for any proper closed subscheme $W$ of $\dim_S(W)\leq q$
in $X$, there exists $X'\in \cSregir$ and a morphism $\pi:X'\to X$ 
such that
\begin{itemize}
\item $\pi$ is surjective and generically finite of degree prime to $\ell$,
\item $W':=\pi^{-1}(W)_{red}$ is an admissible simple normal crossing divisor on $X'$
(cf. Definition \ref{defSNCDadm}).
\end{itemize}

\begin{rem}\label{rem.GG}
\begin{itemize}
\item[(1)]
If the base scheme $S$ is the spectrum of a field of characteristic $0$,
$\bf{(G)}_{\ell,\it{q}}$ holds for all $q$ and $\ell$ thanks to  Hironaka \cite{H}.
\item[(2)]
$\bf{(G)}_{\ell,\it{q}}$ holds for $q\leq 2$ due to \cite[Corollary~0.4]{CJS}. 
Indeed in the cases (1) and (2), one can take $\pi$ to be a birational morphism.
\item[(3)]
If $\ell$ is invertible on $S$, $\bf{(G)}_{\ell,\it{q}}$ for all $q\geq 0$
follows from stronger statements which have been recently proved by Gabber 
(see \cite[Theorem 1.3 and 1.4]{Il2})
\end{itemize}
\end{rem}
\medbreak

Now we state the main theorem.

\begin{theo}\label{mainthm}
Let $H$ be a homology theory leveled above $e$.
Fix an integer $q\geq0$ and a prime $\ell$, and assume the following conditions:
\begin{itemize}
\item[(1)]
$H_a(X)$ is $\ell$-primary torsion for all $X\in \cCS$ and $a\in \Z$,
\item[(2)]
$H$ satisfies the Lefschetz condition
(cf. Definition \ref{def.Lcondition}),
\item[(3)]
the condition $\bf{(PB)}$ from Section~\ref{pullback} holds for $H$,
\item[(4)]
the condition $\bf{(G)}_{\ell,\it{q-2}}$ holds.
\end{itemize}
Then, for any $X\in \cSregir$, the trace map \eqref{KCtrace0} induces a quasi-isomorphism
\[
\tau_{\leq q} \KC X \isom \tau_{\leq q}\KC {T_X}.
\] 
\end{theo}

Before coming to the proof of the theorem we deduce the following corollary generalizing the theorem to not necessarily projective
schemes.

\begin{coro}\label{mainthm.cor}
Let $H$ be a homology theory satisfying (1), (2) and (3) of Theorem~\ref{mainthm}.
Assume further that $\ell$ is invertible on $S$ (so that (4) is satisfied for all 
$q$ by Gabber's theorem).
Let $X\in \cCS$ be regular, proper and connected over $S$ and let $T$ be the image 
of $X$ in $S$. Then $\KH a X=0$ for $a\geq i_T+1$ (cf. Lemma \ref{log.rediso}).
\end{coro}
\begin{proof}
Since $\ell$ is invertible on $S$, we may use Chow's lemma \cite[Section~5.6]{EGAII} and Gabber's theorem \cite[Theorem 1.3 and 1.4]{Il2} to construct $X'\in \cSregirT$ and 
a morphism $\pi: X'\to X$
such that $\pi$ is surjective and generically finite of degree prime to $\ell$. 
By Theorem~\ref{mainthm} we get $\KH a {X'}=0$ for $a\geq i_T+1$. 
From $\bf{(PB)}$ we deduce that $\pi_*:KH_a(X') \to KH_a(X)$
is surjective for $a\in \Z$ and thus the corollary follows.
\end{proof}

\medbreak

For the proof of the theorem, we first show the following Proposition 
\ref{mainprop.geo}. Let $q\geq 1$ be an integer. For a log-pair $\Phi=(X,Y;U)$ 
with $X\in \cSregir$, consider the condition:
\medbreak
$\bf{(LG)}_{\it{q}}$: 
The composite map
$$
\pedgehom a \Phi:H_{a+e}(U)\rmapo{\edgehom q U} \KH a U \rmapo{\partial} 
\KH {a-1} Y
$$
is injective for $a=q$ and surjective for $a=q+1$.
\medbreak

\begin{prop}\label{mainprop.geo}
Let $H$ be a homology theory leveled above $e$. 
Let $X\in \cSregir$ and $d=\dim_S(X)$. Fix a prime $\ell$ and an integer $q\geq 0$.
Assume the following conditions:
\begin{itemize}
\item[(1)]
$\KH q X$ is $\ell$-primary torsion.
\item[(2)]
The condition $\bf{(PB)}$ from Section~\ref{pullback} holds for $H$.
\item[(3)]
The condition $\bf{(G)}_{\ell,\it{q-2}}$ holds.
\item[(4)]
For any log-pair $\Phi=(X',Y')$ with a proper surjective morphism $\pi:X'\to X$
such that $X'\in \cSregir$ and $\dim_S (X')=d$, 
there exists a log-pair $\Phi'=(X',Y'')$ such that $Y'\subset Y''$ and that 
$\Phi'$ satisfies the condition $\bf{(LG)}_{\it{q}}$.
\end{itemize}
Then we have $ \KH q X =0$.
\end{prop}
\bigskip

For the proof of the proposition, we need the following:

\begin{lem}\label{lem1.mainprop}
Let $q\geq 1$ be an integer. 
Let $\Phi=(X,Y;U)$ with $X\in \cSregir$ be a log-pair which satisfies the condition $\bf{(LG)}_{\it{q}}$.
Let $j^*: \KH q X \to \KH q U$ be the pullback via $j:U\hookrightarrow X$ 
and $\edgehom q U: H_{q+e}(U) \to \KH q U$ be as in
Definition \ref{def.Kato.complex}(3).
Then $ (j^*)^{-1}(\Image(\edgehom q U))=0$.
\end{lem}

\begin{proof}
First we claim that $j^*$ is injective. Indeed we have the exact sequence
$$
\KH {q+1} U \rmapo{\partial} \KH q Y \to \KH q X \rmapo{j^*} \KH q U.
$$
Since $\partial \edgehom {q+1} U$ is surjective by the assumption,
$\partial$ is surjective and the claim follows. 
By the above claim it suffices to show 
$\Image(j^*)\cap \Image(\edgehom q U)=0$.
We have the exact sequence
$$
\begin{CD}
\KH q X @>{j^*}>> \KH q U @>{\partial}>> \KH {q-1} Y \\
\end{CD}
$$  
Let $\beta\in H_{q+e}(U)$ and assume $\alpha=\edgehom q U(\beta)\in \KH q U$
lies in $\Image(j^*)$. 
It implies $\partial(\alpha)=\partial\edgehom q U(\beta)=0$.
Since $\partial\edgehom {q} U$ is injective by the assumption,
this implies $\beta=0$ so that $\alpha=0$.
\end{proof}

Now we prove Proposition \ref{mainprop.geo}.

\begin{proof}
Let $\alpha \in KH_q(X)$. By recalling that
$$\edgehom q X: H_{q+e}(X)\to \KH q X=\EX q e 2$$
is an edge homomorphism and by looking at the differentials
$$
d^r_{q,e}: \EX q e r \to \EX {q-r}{r+e-1} r,
$$ 
we conclude that there exists a closed subscheme $W\subset X$ such that
$\dim(W)\leq q-2$, and that putting $U=X-W$, the pullback 
$\alpha_{|U}\in KH_q(U)$ of $\alpha$ via $U\to X$ lies in
the image of $\edgehom q U$, namely there exists $\beta\in H_{q-1}(U)$ such that
$\alpha_{|U}=\edgehom q U(\beta)$. By $\bf{(G)}_{\ell,\it{q-2}}$ we can find 
$X'\in \cSregir$ and a proper surjective map $\pi: X'\to X$ generically finite of
degree prime to $\ell$ such that  $(X',Y')$ is a log-pair where $Y'=\pi^{-1}(W)_{red}$. Put $U'=\pi^{-1}(U)$.
By the condition \ref{mainprop.geo}(4), there is a log-pair $\Phi=(X',Y'';V)$ 
with $Y'\subset Y''$ which satisfies the condition $\bf{(LG)}_{\it{q}}$.
Thanks to the condition \ref{mainprop.geo}(2), we have the commutative diagram
(since $X$ and $X'$ are projective over $S$, $\pi$ is projective and hence it is
embeddable in the sense of Section \ref{pullback})
$$
\begin{CD}
KH_q(X) @>{\pi^*}>> KH_q(X') \\
@VVV @VVV \\
KH_q(U) @>{\pi^*}>> KH_q(U') @>>> KH_q(V) \\
@AA{\epsilon^q_U}A  @AA{\epsilon^q_{U'}}A @AA{\epsilon^q_V}A \\
H_{q+e}(U) @>{\pi^*}>> H_{q+e}(U') @>>> H_{q+e}(V) \\
\end{CD}
$$
\medbreak\noindent
Put $\alpha'=\pi^*(\alpha)\in \KH q {X'}$ and 
$\beta'=\pi^*(\beta)\in H_{q-1}(U')$.
Let $\alpha'_{|V}\in \KH q V$ (resp. $\beta'_{|V}\in H_q(V)$) 
be the pullback of $\alpha'$ (resp. $\beta'$) via 
$V\hookrightarrow X'$ (resp. $V\hookrightarrow U'$). 
By the diagram we get 
$\alpha'_{|V}=\edgehom q V(\beta'_{|V}) \in \KH q V$.
By Lemma \ref{lem1.mainprop} this implies $\alpha'=0$. Since the composite
$$
\KH q X \rmapo{\pi^*} \KH q {X'} \rmapo{\pi_*} \KH q X
$$
is the multiplication of the degree of $\pi$ which is prime to $\ell$,
we get $\alpha=0$ by the condition \ref{mainprop.geo}(1).
\end{proof}

For a closed subscheme $T\subset S$ and integers $q,d\geq 0$, 
consider the condition:

\medbreak\noindent
${\bf{KC}}_T(q,d)$: For any $X\in \cSregir$ with $\dim(X)\leq d$ and $T_X\subset T$, where
$T_X$ is the image of $X$ in $S$, 
the proper map $X\to T_X$ induces a quasi-isomorphism
$\tau_{\leq q} \KC X \isom \tau_{\leq q} \KC {T_X}$.
\medbreak

\begin{lem}\label{mainlem}
Assume ${\bf{KC}}_T(q,d-1)$ for integers $d\geq 1$ and $q\geq 0$. 
If a log-pair $\Phi=(X,Y;U)$ with $X\in \cSregirT$ is $H$-clean in degree $q$
and if $q\geq i_T+1$ (cf. Definition~\ref{def.Lcondition} and Lemma~\ref{log.rediso}), 
then it satisfies $\bf{(LG)}_{\it{q}}$.
\end{lem}
\noindent
\begin{proof}
We consider the commutative diagram
$$
\begin{CD}
H_{q+1+e}(U)@>{\edgehom {q+1} U}>>\KH {q+1} U @>{\partial}>> \KH q Y \\
@. @VV{\graphhom {q+1} \Phi}V @V{\simeq}V{\graphhom {q} \hPhi}V \\
@. \graphHempty {q+1} \Phi @>{\simeq}>{\partial}>  \graphHempty {q} \hPhi \\
\end{CD}
$$  
where $\graphhom {q} \hPhi$ is an isomorphism by ${\bf{KC}}_T(q,d-1)$
and Lemma \ref{lem.reduced.graphhom}, and the lower $\partial$ is an isomorphism 
by Lemma~\ref{log.rediso} (use $q\geq \dim_S(T)+1$),
and $\graphedge {q+1} \Phi=\graphhom {q+1} \Phi\circ \edgehom {q+1}U$ is surjective by the cleanness assumption. This shows $\partial\edgehom {q+1} U$ is surjective.
Next we consider the commutative diagram 
$$
\begin{CD}
 H_{q+e}(U) @>{\edgehom q U}>> \KH q U @>{\partial}>> \KH {q-1} Y \\
 @.@VV{\graphhom q \Phi}V @V{\simeq}V{\graphhom {q-1} \hPhi}V \\
@.\graphHempty q \Phi @>{\hookrightarrow}>{\partial}>  \graphHempty {q-1} \hPhi \\
\end{CD}
$$  
where $\graphhom {q-1} \hPhi$ is an isomorphism by ${\bf{KC}}_T(q,d-1)$ and 
Lemma \ref{lem.reduced.graphhom}, and the lower $\partial$ is injective by 
Lemma~\ref{log.rediso}, and
$\graphedge {q} \Phi=\graphhom {q} \Phi\circ \edgehom {q}U$ is injective by 
the cleanness assumption. This shows $\partial\edgehom {q} U$ is injective and the proof is complete.
\end{proof}

\medskip

We now prove Theorem \ref{mainthm}. 

\begin{proof}
For this we may fix a closed subscheme $T\subset S$. It suffices to show ${\bf{KC}}_T(q,d)$ holds for all $d\geq 0$.
By Lemma \ref{lem0.Lefschetz} we may assume $q\geq \dim_S(T)+1\geq i_T+1$.
We now proceed by induction on $d=\dim(X)$. The case $d=0$ follows from Lemma 
\ref{lem0.Lefschetz}. Assume $d\geq 1$ and ${\bf{KC}}_T(q,d-1)$ and we want to 
show $\KH q X=0$ for $X\in \cSregirT$. For this we apply Proposition \ref{mainprop.geo} to $X$. The conditions (1), (2) and (3) hold by the assumption of
the theorem. The condition (4) of the proposition is satisfied by Lemma \ref{mainlem} and
the Bertini Theorem~\ref{Bertini} due to the Lefschetz condition.
This completes the proof of the theorem.
\end{proof}


\section{Result with finite coefficients}\label{finitecoeff}

\bigskip

\noindent
Theorem \ref{mainthm} shows the vanishing of the Kato homology of an object of $\mathcal{S}_{reg}$ for a homology theory satisfying the Lefschetz condition.
By Theorem \ref{thm.Lcondition} the condition is satisfied for some homology 
theories with admissilbe coefficient $\Lambda=T\otimes_{\Z_{\ell}} \Q_{\ell}/\Z_{\ell}$
as in Definition \ref{admisdefi}. In this section we improve Theorem~\ref{mainthm} and Corollary~\ref{mainthm.cor} to the case of
finite coefficient $\Ln=T\otimes_{\Z_{\ell}} \lnz$.
\medbreak

Fix a prime $\ell$ and assume given an inductive system of homology theories:
$$H=\{\HLLn\}_{n\geq 1},$$
where $H(-,\Ln)$ are homology theories leveled above $e$ on $\cCS$.
We assume $H(X,\Ln)$ is killed by $\ell^n$ for any $X\in \cCS$. 
It gives rise to a homolgy theory (again leveled above $e$) on $\cCS$:
$$ H(-,\Linfty)\;:\; X \to H_{a}(X,\Linfty):=\indlim {n\geq 1} H(X,\Ln) 
\qfor X\in Ob(\cCS)$$
with 
$\iota_n: H(-,\Ln) \to H(-,\Linfty)$, a natural transformation of homology theories.
We further assume given for each integer $n\geq 1$, 
a map of homology theories of degree $-1$
\begin{equation}\label{boundaryL}
\partial_n\;:\; H(-,\Linfty) \to H(-,\Ln)
\end{equation}
such that for any $X\in \cC$ and for any integers $m>n$,
we have the following commutative diagram of exact sequences
\begin{equation}\label{ESfinite}
\begin{CD}
0 @>>> \Hinfty {q+1} X/\nt @>{\partial_n}>> \HLn q X 
@>{\iota_n}>> \Hinfty q X [\nt] @>>> 0 \\
@. @| @VV{\iota_{m,n}}V @|\\
0@>>> \Hinfty {q+1} X/\mt @>{\partial_m}>> \HLm q X 
@>{\iota_m}>> \Hinfty q X [\mt] @>>> 0. \\
\end{CD}
\end{equation}
\medbreak

We let $\KCn X$ and $\KCinfty X$ denote the Kato complexes
associated to $H(-,\Ln)$ and $H(-,\Linfty)$ respectively and 
$\KHn a X$ and $\KHinfty a X$ denote their Kato homology groups.
By definition 
$\KHinfty a X =\indlim {n\geq 1} \KHn a X.$ 

\begin{rem}
The above assumption is satisfied for the inductive systems of homology theories 
$\{\H^{\et}(-,\Ln)\}_{n\geq 1}$ and 
$\{\H^{D}(-,\Ln)\}_{n\geq 1}$ in Examples~\ref{exHK1} and \ref{exHK2}, where
$\Ln=T\otimes \lnz$ for a finitely generated free $\zl$-module with
continuous $G_k$-action. 
\end{rem}
\medbreak

For an integer $q\geq 0$ we consider the following condition for $H(-,\Linfty)$:

\begin{enumerate}
\item[$\mathrm{(} \mathbf D \mathrm{)}_q$] : 
For any $x\in \Xd q$, $\Hinfty {q+e+1} x$ is divisible.
\end{enumerate}

\begin{rem}\label{finitecoeff.rem2}
\par\noindent
\begin{itemize}
\item[(1)]
For $H=\{\H^{\et}(-,\lnz)\}_{n\geq 1}$, the condition 
{$\mathrm ( \mathbf D \mathrm )_q$} is implied by the Bloch-Kato conjecture
(see Lemma \ref{lem.Dq} below).
For $H=\{\H^{D}(-,\lnz)\}_{n\geq 1}$, it is not known to hold.
\item[(2)]
Let $x\in \Xd q$.
In view of \eqref{ESfinite} and the vanishing of $\Hinfty {q+e-1} x$ 
(which follows from the assumption that $H(-,\Ln)$ is leveled above $e$),
{$\mathrm ( \mathbf D \mathrm )_q$} is equivalent to the exactness of the following sequence
\[
0\to \HLn {q+e} {x} \rmapo{\iota_n} \Hinfty {q+e} {x}\rmapo{\ell^n}
\Hinfty {q+e} {x} \to 0.
\]
\end{itemize}
\end{rem}
\medbreak

\begin{lem}\label{lem.Dq}
For $H=H^\et(-,\lnz)$ in Example~\ref{exHK1}, 
{$\mathrm ( \mathbf D \mathrm )_q$} holds.
\end{lem}
\begin{proof}
In view of Lemma \ref{lemHK1smooth}, $\mathrm ( \mathbf{D} \mathrm )_q$ is equivalent to 
the condition that 
\[
\indlim n H^q_{\et}(\Spec(L),\lnz(q))
\]
is divisible for the residue field $L=\k(x)$ of $x\in \Xd q$.
This is a consequence of the surjectivity of the Galois symbol map
(\cite{Mi} and \cite[Section 2]{BK}):
$$h^q_{L} : K^M_q(L) \to H^q_{et}(\Spec(L),\lnz(q))$$
where $K^M_q(-)$ denotes the Milnor $K$-group. 
In case $\ell =\ch(L)$ the surjectivity had been shown by Bloch-Gabber-Kato \cite{BK}.
In case $\ell\not=\ch(L)$ it has been called the Bloch-Kato conjecture and now
established by Rost and Voevodsky (the whole proof is available in the series of 
the papers \cite{V1}, \cite{V2}, \cite{SJ} and \cite{HW}).
\end{proof}
\medbreak

\begin{theo}
\label{mainthm.finite}
Let $H$ be an inductive system of homology theories as above.
Fix an integer $q\geq0$ and assume the following conditions:
\begin{itemize}
\item[(1)]
$H(-,\Lam_\infty )$ satisfies the Lefschetz condition
(cf. Definition \ref{def.Lcondition}).
\item[(2)]
$H(-,\Lam_\infty )$ satisfies $\bf{(PB)}$ from Section~\ref{pullback}.
\item[(3)]
$\bf{(G)}_{\ell,\it{q-1}}$ from Section~\ref{maintheo} holds
(this is the case if $q\leq 3$ by \cite{CJS}). 
\item[(4)]
${\bf(D)}_q$ and ${\bf(D)}_{q-1}$ hold.
\end{itemize}
Then, for $X\in \cS^{irr}_{reg}$ with image $T$ in $S$, 
$\KHn q X=0$ if $q\geq i_T+1$.
\end{theo}
\medbreak

\begin{coro}\label{mainthmcor.finite}
Assume that $H(-,\Linfty)$ as above satisfies the Lefschetz condition
and $\mathrm ( \bf{PB} \mathrm )$ from Section~\ref{pullback}.
Assume further that $\ell$ is invertible on $S$.
Let $X\in \cCS$ be regular, proper and connected over $S$ and let $T$ be the image 
of $X$ in $S$. Let $q\geq 1$ be an integer. 
Assuming {$\mathrm ( \mathbf D \mathrm )_q $ } and {$ \mathrm ( \mathbf D \mathrm )_{q-1} $}, we have
$\KHn q X=0$ if $q\geq i_T+1$. 
\end{coro}
\medbreak

The theorem and corollary follow at once from Theorem~\ref{mainthm}, Corollary~\ref{mainthm.cor} and the following:

\begin{lem}
If {$ \mathrm ( \mathbf D \mathrm )_q $} and {$\mathrm ( \mathbf D \mathrm )_{q-1}$} holds, we have an exact sequence for $X\in\cCS$:
\[
0\to {\KHinfty {q+1} X}/\ell^n \to  \KHn q X \to {\KHinfty q X}[\ell^n]\to 0.
\]
\end{lem}

\begin{proof}
This follows from the following commutative diagram:
\[
\begin{CD}
@.0@.0@.\\
@. @VVV @VVV  \\
 \sumd X {q+1} \HLn {q+e+1} x @>>> \sumd X {q} \HLn {q+e} x @>>>
\sumd X {q-1} \HLn {q+e-1} x \\
@VV{\iota_n}V @VV{\iota_n}V @VV{\iota_n}V  \\
 \sumd X {q+1} \Hinfty {q+e+1} x @>>> \sumd X {q} \Hinfty {q+e} x @>>>
\sumd X {q-1} \Hinfty {q+e-1} x \\
@VV{\ell^n}V @VV{\ell^n}V @VV{\ell^n}V  \\
 \sumd X {q+1} \Hinfty {q+e+1} x @>>> \sumd X {q} \Hinfty {q+e} x @>>>
\sumd X {q-1} \Hinfty {q+e-1} x \\
@VVV @VVV @VVV  \\
0@.0@.0@.\\
\end{CD}
\]
where all columns and rows are exact by the assumption and Remark \ref{finitecoeff.rem2}(2).
\end{proof}
\bigskip

In the above argument, in order to show the vanishing of $\KHn q X$,
the vanishing of $\KHinfty {q} X$ and $\KHinfty {q+1} X$ are used.
In \cite[Section 5]{JS2} a refined argument is given which uses only the vanishing 
of $\KHinfty {q} X$ to show that of $\KHn q X$.
This is useful when we deal with the case $\ell=p$ for which we cannot use
Gabber's theorem and obliged to resort to results on resolution of singularities
in low dimension such as \cite{CJS}. In what follows we recall the argument.
\medbreak

\begin{theo}\label{mainthm.finite}
Let $q\geq 1$ be an integer. Assume the following:
\begin{itemize}
\item[(1)]
$H(-,\Linfty)$ satisfies the Lefschetz condition.
\item[(2)]
$\HLn \empty -$ satisfies $\bf{(PB)}$ from Section~\ref{pullback}.
\item[(3)]
$\bf{(G)}_{\ell,\it{q-2}}$ from Section~\ref{maintheo} holds
(this is the case if $q\leq 4$ by \cite{CJS}). 
\item[(4)]
{${\bf(D)}_q$} holds.
\item[(5)]
We have an exact sequence of the Kato complex for $S$:
\[
0\to \KCn S \rmapo{\iota_n} \KCinfty S \rmapo{\ell^n} \KCinfty S\to 0.
\]
\end{itemize}
Then, for $X\in \cS^{irr}_{reg}$ with image $T$ in $S$, 
$\KHn q X=0$ if $q\geq i_T+1$.
\end{theo}
\begin{proof}
We may assume $\dim_S(X) \geq q$. In case $\dim_S(X)=q$, by Remark \ref{finitecoeff.rem2}(2), {${\bf(D)}_q$} implies that $\KHn q X$ is a subgroup of 
$\KHinfty q X$. Hence the assertion follows from Theorem \ref{mainthm}. 
Assume $\dim_S(X)> q$. We apply to $H=\HLLn$ the same argument as the proof of 
Theorem \ref{mainthm} (reducing it to Proposition \ref{mainprop.geo} and 
Lemma \ref{mainlem}). We are reduced to showing the following:
\end{proof}

\begin{lem}\label{mainthm.finite.lem}
Let $\Phi=(X,Y;U)$ be a log-pair with $\dim_S(X)\geq q+1$.
If $\Phi$ is clean in degree $q+1$ and $q$ for $\Hinfty \empty -$, 
it is clean in degree $q$ for $H=\HLLn$.
\end{lem}
\begin{proof}
We have the commutative diagram {
$$
\begin{CD}
0 @>>> \Hinfty {a+1} U /\nt @>{\partial_n}>> \Hn a U @>>> \Hinfty a U [\nt]
 @>>> 0 \\
@. @VV{\graphedge {a+1}{\Phi,\Linfty}}V @VV{\graphedge {a}{\Phi,\Ln}}V 
@VV{\graphedge {a}{\Phi,\Linfty}}V \\
0@>>> \graphHinfty {a+1} {\Phi}/\nt @>>> \graphHn a {\Phi} @>>> \graphHinfty a {\Phi}[\nt] @>>> 0\\
\end{CD}
$$ 
}
where the exactness of the lower sequence follows from the condition (5) of 
Theorem \ref{mainthm.finite}.
The assumption implies
$\graphedge {a}{\Phi,\Linfty}$ is injective for $a=q+1$ and $q$ 
(resp. surjective for $a=q+2$ and $q+1$). By the diagram this implies that
$\graphedge {a}{\Phi,\Ln}$ is injective for $a=q$ and surjective for $a=q+1$, 
which completes the proof of the lemma.
$\square$
\end{proof}


\section{Kato's Conjectures}\label{kato}

\bigskip

\noindent 
In this section we prove part of the original conjectures of Kato from~\cite{K}. 
We fix a prime $\ell$ and consider the Kato homology with $\ell$-primary torsion 
coefficient.

We start with Kato's Conjectures~0.3 and 5.1 from~\cite{K}.
Assume either of the following:
\begin{itemize}
\item[$(i)$]
$S=\Spec(k)$ for a finite field $k$,
\item[$(ii)$]
$S=\Spec(R)$ where $R$ is a henselian discrete valuation ring with finite residue field.
\end{itemize}
In case $(ii)$ we assume that $\ell$ is invertible in $R$. 
Let $G=\Gal(\overline{k}/k)$ be the absolute Galois group of $k$ in case $(i)$ and
$G=\pi_1(S,\overline{\eta})$ with a geometric point $\overline{\eta}$ over
the generic point $\eta$ of $S$ in case $(ii)$.
We fix an $\ell$-primary torsion $G$-module $\Lam$ and consider the following
\'etale homology theory on $\cCS$: In case $(i)$ it is given by
$$
H^\et_a(X,\Lam)=H^{-a}(X_{\et}, R\,f^{!}\Lam)
\qfor f:X\rightarrow S \text{ in } \cCS.
$$
In case $(ii)$ it is given by
$$
H^\et_a(X,\Lam)=H^{2-a}(X_{\et}, R\,f^{!}\Lam(1))
\qfor f:X\rightarrow S \text{ in } \cCS.
$$
This is leveled above $-1$ and the associated Kato complex is written as:
\begin{multline}\label{KCfinitefield2}
\cdots \sumd X a H_{\et}^{a+1}(x,\Lam(a))\to
\sumd X {a-1} H_{\et}^{a}(x,\Lam(a-1))\to \cdots \\
\cdots \to\sumd X 1 H_{\et}^{2}(x,\Lam(1))\to 
\sumd X 0 H_{\et}^{1}(x,\Lam).
\end{multline}
For $\Lambda=\lmz$ (with the trivial $G$-action), this coincides with
the complex \eqref{eq.KC1} (with $n=\ell^m$) which Kato considered originally.
Let $\KHetL a X$ be the associated Kato homology groups.
The following theorem gives a complete answer to Kato's Conjectures~0.3 and 5.1
from~\cite{K} for $\ell$ invertible on $S$.

\begin{theo}\label{kato.finitefield}
Assume either of the following:
\begin{itemize}
\item[(1)]
$\Lam=T\otimes\qzl$ is admissilbe in the sense of Definition \ref{admisdefi}.
\item[(2)]
$\Lambda=\lnz$ with the trivial $G$-action.
\end{itemize}
For a regular connected scheme $X$ proper over $S$ with image $T_X$ in $S$ we have
\[
\KHetL a X = \left\{ \begin{array}{ll} 
\Lam_{G} & \text{ for } a=0 \text{ and } \dim (T_X)=0 \\
0 & \text{ otherwise }.
\end{array}
 \right.
\]
under one of the following assumptions:
\begin{itemize}
\item { $\ell$} is invertible on $S$,
\item $X$ is projective and { ${\bf (G)}_{\ell,a-2}$} holds (which is the case if
$a\leq 4$ by \cite{CJS}).
\end{itemize}
\end{theo}
\begin{proof}
Under the assumption (1) the theorem follows from 
Theorem \ref{mainthm}, Corollary \ref{mainthm.cor}, Theorem \ref{thm.Lcondition}
and Theorem \ref{PBthm}. We show the theorem under the assumption (2) by applying
Theorem \ref{mainthm.finite} to the inductive system of the \'etale homology theories:
\[
H=\{\Het(-,\lnz)\}_{n\geq 1}.
\]
The condition (5) is easily checked.
The condition (1) follows from Theorem \ref{thm.Lcondition} and 
(2) from Theorem \ref{PBthm}.  Finally (4) follows from Lemma \ref{lem.Dq}.
This completes the proof of Theorem \ref{kato.finitefield}
\end{proof}

\medskip

Another consequence of our approach to Kato's conjectures, which was not directly conjectured by Kato himself and which will be used in \cite{KeS} to study quotient singularities, is the following theorem.
Note that we do not know at present whether it holds with finite
coefficients except in case $k$ is 
a one-dimensional global field (see the proof of 
Theorem \ref{kato.globalfield}).

Let $S=\Spec(k)$ where $k$ is a finitely generated field $k$ or its purely inseparable extension. 
Let $G=\Gal(\overline{k}/k)$ be the absolute Galois group of $k$.
We fix an $\ell$-primary torsion $G$-module $\Lam$ and consider the
following \'etale homology theory on $\cCS$ from Example~\ref{exHK2}:
$$
\HDL a X:=\indlim F \Hom\big(H^{a}_c(X_{\et},F^\vee),\qz\big) \qfor X\in Ob(\cCS),
$$ 
where $F$ runs over all finite $G$-submodules of $\Lambda$ and $F^\vee=\Hom(F,\qz)$.
This homology theory is leveled above $0$.
Let $\KHDL a -$ be the associated Kato homology groups.

\begin{theo}\label{kato.finitegenfield}
Assume $\Lam=T\otimes\qzl$ is admissilbe in the sense of Definition \ref{admisdefi}.
For a connected scheme $X$ smooth proper over $k$, we have
\[
\KHDL a X  = \left\{ \begin{array}{ll} 
\Lam_G & \text{ for } a=0  \\
0 & \text{ for } a>0.
\end{array}
 \right.
\]
under one of the following assumptions:
\begin{itemize}
\item { $\ell\ne \ch(k)$},
\item $X$ is projective and { ${\bf (G)}_{\ell,a-2}$} holds after replacing the base field $k$ by its perfection.
\end{itemize}
\end{theo}
\begin{proof}
Let $\tk$ be the perfection of $k$.
From Lemma \ref{lemHK1smooth2}(3) we deduce a canonical isomorphism
\[
\KHDL a {X\otimes_k \tk} \isom \KHDL a X.
\]
Hence we may assume that $k$ is perfect. Then the proof is the same as that of Theorem~\ref{kato.finitefield}. 
$\square$
\end{proof}
\medskip

Next we discuss Conjecture~0.4 from \cite{K}. 
Let $S=\Spec(K)$ where $K$ is a one-dimensional global field.
Let $P_K$ be the set of all places of $K$. 
Let $\cCS$ be the category of schemes separated and of finite type over $S$.
Let $\ell$ be a prime number with $\ell\ne \chara (K)$ and write $\Ln=\lnz$
for an integer $n>0$.
For $X\in \cCS$ with $\dim(X)=d$, Kato considered the complex { $KC^{(1)}(X,\Ln)$}:
\[
\bigoplus_{x\in \Xd {d}} H^{d+2}_{\et}(x,\Ln(d+1) ) \to 
\bigoplus_{x\in \Xd {d-1}} H^{d+1}_{\et}(x,\Ln(d)) \to \cdots \to 
\bigoplus_{x\in \Xd {0}} H^{2}_{\et}(x,\Ln(1))
\]
and the corresponding complexes $KC^{(1)}(X_{K_v},\Ln)$ for 
$X_{K_v}=X\times_K K_v$ with the henselization $K_v$ of $K$ at $v\in P_K$. 
Here we put the the sum $\oplus_{x\in \Xd {a}}$ at degree $a$ for $a\ge 0$. 
Let us define the complex {
\begin{equation}
\KCcLn X := \mathrm{cone} [ KC^{(1)} (X,\Ln) \to
\underset{v\in P_K}{\bigoplus} \; KC^{(1)}(X_{K_v},\Ln ) ] 
\end{equation}  
}
and its homology
\begin{equation}
\KHcLn a X = H_a(\KCcLn X)\qfor a\in \bZ.
\end{equation} 
In fact in his original approach Kato considered completions instead of henselizations 
of $K$, but Jannsen~\cite{J3} showed that this does not change homology. 
The next theorem is due to Jannsen in case $K$ is a number field~\cite{J2}
and gives a partial answer to Kato's Conjecture~0.4 \cite{K}.

\begin{theo}\label{kato.globalfield}
For a connected scheme $X$ proper smooth over the one-dimensional global field $K$ and for { $\ell\ne \chara(K)$ } we have
\[
\KHcLn a X = \left\{ \begin{array}{ll} 
\Ln & \text{ for } a=0  \\
0 & \text{ for } a>0.
\end{array}
 \right.
\]
\end{theo}
\begin{proof}
Let $H=H^D(-,\Lam_n)$ be the homology theory from Example~\ref{exHK2} over the base 
$S=\Spec(K)$ and $KC_{H}(X,\Lam_n )$ be the associated Kato complex. 
We claim that for any $Y\in \cCS$ there is an exact sequence 
\begin{equation}\label{jannsen.kato}
0 \to   KC^{(1)}(Y,\Lam_n) \rmapo{\iota} 
\underset{v\in P_K}{\bigoplus} \; KC^{(1)}(Y_{K_v},\Ln )
 \rmapo{\rho}  KC_{H}(Y,\Lam_n ) \to 0.
\end{equation}
Indeed, the injectivity of $\iota$ is due to Jannsen \cite[Theorem 0.2]{J2}
and a generalization of \cite[Theorem 0.4]{J2} using Gabber's refinement of 
de Jong's alterations { \cite{Il2}}.
The map $\rho$ as well as the right-exactness of \eqref{jannsen.kato} comes from the duality stated Corollary~\ref{coro.compact}.
\eqref{jannsen.kato} implies that we have an isomorphism of Kato homology groups
\[
\KHDLn a X \cong \KHcLn a X\qfor a\in \bZ.
\]
Thus we are reduced to show the assertion of the theorem by replacing $\KHcLn a X$ by $\KHDLn a X$.
By Lemma \ref{lemHK1smooth2}(3) we may replace $K$ by its perfection. 
Then we can apply Theorem \ref{mainthm.finite} to the inductive system of the homology theories:
\[H=\{H^D(-,\Ln)\}_{n\geq 1}.
\]
We have already checked all the conditions other than (4).  
As for the last condition we have to show the injectivity of
$\HDLn q x \to \HDLinfty q x$ for $x\in \Xd q$ (cf. Remark \ref{finitecoeff.rem2}(2)).
By \eqref{jannsen.kato} we have an exact sequence
\[
0\to H^{q+2}_{\et}(x,\lnz(q+1))\to \underset{v\in P_K}{\bigoplus} \;
H^{q+2}_{\et}(x_v,\lnz(q+1)) \to \HDLn q x \to 0,
\]
where $x_v=x\times_{\Spec(K)}\Spec(K_v)$. Hence it suffices to show the injectivity of
\[
H^{q+2}_{\et}(\Spec(L),\lnz(q+1))\to H^{q+2}_{\et}(\Spec(L),\qzl(q+1))
\]
for residue fields $L$ of $x$ and $x_v$, which follows from the divisibility of
\[H^{q+1}_{\et}(\Spec(L),\qzl(q+1)).\]
Thus the desired assertion follows from 
the Bloch-Kato conjecture (see the proof of Lemma \ref{lem.Dq}).
This completes the proof of Theorem \ref{kato.globalfield}.
\end{proof}

\medskip

Finally, we discuss Kato's Conjecture~0.5 from~\cite{K}. Let $K$ be a one-dimensional global field. Fix a prime $\ell\not=\ch(K)$ and put $\Ln=\lnz$ for an integer $n>0$.
Fix a connected regular scheme $U$ of finite type over $\Z$ with the function field $K$. 
For a scheme $X$ separated and of finite type over $U$, let { $KC^{(0)}(X,\Lam_n)$ }
denote the complex \eqref{KCfinitefield2}:
\[
\bigoplus_{x\in X_{(d)}} H^{d+1}_{\et}(x,\Lam_n (d) ) \to \bigoplus_{x\in X_{(d-1)}} H^{d}_{\et}(x,\Lam_n(d-1)) \to \cdots \to \bigoplus_{x\in X_{(0)}} H^{1}_{\et}(x,\Lam_n ).
\]
We have the natural restriction map {
\[
KC^{(0)}(X,\Lam_n)[1] \to KC^{(1)}(X_K,\Lam_n)
\qwith X_K=X\times_U\Spec(K).
\]
}
and one considers a version of Kato complex defined as {
\[
\KCccLn X :=\mathrm{cone} [ KC^{(0)}(X,\Lam_n )[1] \to \bigoplus_{v\in \Sigma_U} KC^{(1)}(X_{K_v} ,\Lam_n ) ]
\]
}
and its homology group
\[
\KHccLn a X =H_a(\KCccLn X) \qfor a\in \bZ.
\]
Here $\Sigma_U$ denotes the set of $v\in P_K$ which do not correspond to 
closed points of $U$ and $X_{K_v}=X\times_U K_v$ with $K_v$, 
the henselization of $K$ at $v$. 

\begin{theo}\label{kato.arith}
Assume that $\ell$ is invertible on $U$. 
For a regular connected scheme $X$ proper and flat over $U$ and with smooth generic fibre
over $U$, we have
\[
\KHccLn a X = \left\{ \begin{array}{ll} 
\Lam_n & \text{ for } a=0  \\
0 & \text{ for } a>0.
\end{array}
 \right.
\]
\end{theo}

\begin{rem}
For $\dim(X)=1$ Theorem~\ref{kato.arith} is a version of the celebrated Brauer-Hasse-Noether Theorem proved in the 1930s. For $\dim(X)=2$ it was shown by 
Kato~\cite{K} and motivated him to conjecture the general case. The case $\dim(X)=3$ can be deduced from the results proved in~\cite{JS1}.
\end{rem}

\begin{proof}
Let $v\in U$ be a closed point and denote by the same letter $v$ 
the associated place of $K$. Letting $X_v=X\times_U v$, there is a residue map of complexes
(see \cite[Proposition 2.12]{JS1})
\[
\delta\;:\; KC^{(1)}(X_{K_v},\Lam_n ) \to KC^{(0)}(X_v,\Lam_n ) 
\]
whose cone is identified up to sign with the complex
\eqref{KCfinitefield2} (with $\Lam=\Ln$) for $X\times_U R_v$ where $R_v$ is the henseliztion of $\cO_{U,v}$. 
Hence $\delta$ is a quasi-isomorphism by Theorem~\ref{kato.finitefield}. 
This implies an isomorphism
\[
\KHcLn a {X_K}  \cong \KHccLn a X  \qfor a\in \bZ
\]
and Theorem~\ref{kato.arith} follows from Theorem~\ref{kato.globalfield}.
\end{proof}


\section{Application to Cycle Maps} \label{application1}

\bigskip

\noindent
In this section $X$ denotes a separated scheme of finite type over a field or 
a Dedekind domain. Let $\CH^r(X,q)$ be Bloch's higher Chow group defined by Bloch 
\cite{B} (see also \cite{Le} and \cite{Ge1}). 
It is related to algebraic $K$-theory of $X$ via the Atiyah-Hirzebruch spectal sequence (\cite{Le}):
\begin{equation}\label{AHss}
E_2^{p,q}= \CH^{-q/2}(X,-p-q;\nz) \Rightarrow K'_{-p-q}(X,\nz).
\end{equation}
In case $X$ is smooth over a field, it is related to the motivic cohomology defined by Voevodsky via
\begin{equation}\label{MCHC}
\HMXr q=  \CH^r(X,2r-q) .
\end{equation}
A `folklore conjecture', generalizing the analogous conjecture of Bass on 
$K$-groups, is that $\CH^r(X,q)$ should be finitely generated if $X$ is over $\bZ$ or 
a finite field. 
Except the case $r=1$ or $\dim(X)=1$ (Quillen), the only general result has 
been known for the Chow group $\CH_0(X)$ of zero-cycles, which is 
a consequence of higher dimensional class field theory 
(\cite{Sa1} and \cite{CTSS}).
\medbreak

One way to approach the problem is to look at 
an \'etale cycle map constructed by Bloch \cite{B}, 
Geisser and Levine \cite{GL2} and Levine \cite{Le} (see also \cite{Ge1}): 
Fix an integer $n>0$.
Let $\CH^r(X,q; \nz)$ be the higher Chow group with finite coefficients, 
which fits into a short exact sequence:
\begin{equation}\label{HCHfinite}
0\to \CH^r(X,q)/n \to \CH^r(X,q; \nz) \to \CH^r(X,q-1)[n] \to 0.
\end{equation}
Assume that $X$ is regular and that
either $n$ is invertible on $X$ or $X$ is smooth over a perfect field. 
Then we have a natural map
\begin{equation}\label{ccmap}
\rho^{r,q}_{X,\nz} \;:\; \CH^r(X,q;\nz) \to H^{2r-q}_{\et}(X,\nz(r)),
\end{equation}
where $\nz(r)$ is the Tate twist introduced in Lemma \ref{lemHK1smooth}.
Sato \cite{Sat} constructed a similar map in case $X$ is flat over $\bZ$ and 
$n$ is not necessarily invertible on $X$ but we will not use this in this paper.
The following theorem follows from the results of 
Suslin-Voevodsky \cite{SV} and Geisser-Levine \cite{GL2} together with
the Bloch-Kato conjecture:

\begin{theo}\label{thm.VS-GL}
The map $\rho_{X,\nz}^{r,q}$ is an isomorphism for $r\leq q$ and injective
for $r=q+1$. 
\end{theo}

\begin{coro}\label{cor.VS-GL}
Let $Z$ be a quasi-projective scheme over either a finite field or $\bZ$
or a henselian discrete valuation ring with finite residue field.
Let $n>0$ be an integer invertible on $X$.
Then $\CH^r(Z,q;\nz)$ is finite for $r\leq q+1$. 
\end{coro}

\begin{proof}
The localization sequence for higher Chow groups implies that 
for a closed subscheme $Y\subset Z$ of pure codimension $c$ with
the complement $V=Z-Y$, we have a long exact sequence
\begin{multline}\label{locseq}
\cdots \CH^{r-c}(Y,q;\nz) \to \CH^{r}(Z,q;\nz)\to \CH^{r}(V,q;\nz)\\
\to \CH^{r-c}(Y,q-1;\nz)\to \cdots
\end{multline}
Hence we may replace $Z$ by $V$ and this reduces the proof to the case where
$Z$ is regular, which follows from Theorem \ref{thm.VS-GL} thanks to the finiteness 
result for \'etale cohomology { $H^*_{\et}(X,\nz(r))$ }
(cf. \cite[(26) on page 780]{CTSS}, \cite[Theorem 1.9]{M1}).
\end{proof}
\bigskip

Now we turn our attention to $\rho_{X,\nz}^{r,q}$ for $r\geq d:=\dim(X)$ and the base scheme is as in Corollary~\ref{cor.VS-GL}.
In case $r>d$ it is easily shown (see \cite[Lemma 6.2]{JS2}) that $\rho_{X,\nz}^{r,q}$ is an isomorphism (the proof uses Theorem \ref{thm.VS-GL}).
An interesting phenomenon emerges for $\rho_{X,\nz}^{r,q}$ with $r=d$.
Using Theorem \ref{thm.VS-GL}, one can show that there exits a long exact sequence 
(see \cite[Lemma 6.2]{JS2}):
\begin{multline}\label{les1}
\KHnz {q+2} X \to \CH^d(X,q;\nz) \rmapo{\rho_{X,\nz}^{d,q}} 
H^{2d-q}_{\et}(X,\nz(d)) \\
\to \KHnz {q+1} X \to \CH^d(X,q-1;\nz) \rmapo{\rho_{X,\nz}^{d,q-1}} \dots
\end{multline}
where $\KHnz * -$ is the Kato homology defined as \eqref{eq.KHintro}.
Hence Theorem \ref{kato.finitefield} implies the following:

\begin{theo}\label{thm.cyclemap}
Assume either of the following conditions:
\begin{itemize}
\item[$(i)$]
$X$ is regular and proper over either a finite field or a henselian 
discrete valuation ring with finite residue field and $n$ is invertible on $X$.
\item[$(ii)$]
$X$ is smooth projective over a finite field and $q\leq 2$.
\end{itemize}
Then we have the isomorphism
$$
\rho^{d,q}_{X,\nz} \;:\; \CH^d(X,q; \nz) \isom H^{2d-q}_{\et}(X,\nz(d)).
$$
In particular $\CH^d(X,q;\nz)$ is finite.
\end{theo}
\medbreak

The above theorem implies the following affirmative result on the finiteness conjecture on motivic cohomology:

\begin{coro}\label{cor.finiteness}
Let $Z$ be a quasi-projective scheme over either a finite field or
a henselian discrete valuation ring with finite residue field.
Let $n>0$ be an integer invertible on $Z$.
\begin{itemize}
\item[(1)]
$\CH^r(Z,q;\nz)$ is finite for all $r\geq \dim(Z)$ and $q\ge 0$. 
\item[(2)]
$K'_i(Z,\nz)$ is finite for $i\geq \dim(Z)-2$.
\end{itemize}
\end{coro}

\begin{proof}
(2) follows from (1) and \eqref{AHss}.
To show (1), we may assume $n=\ell^m$ for a prime $\ell$ invertible on $Z$.
We proceed by the induction on $\dim(Z)$. First we remark that the localization 
sequence \eqref{locseq} implies that for a dense open subscheme $V\subset Z$,
the finitenss of $\CH^r(Z,q;\nz)$ for all $r\geq \dim(Z)$, is equivalent
to that of $\CH^r(V,q;\nz)$. In particular we may suppose $Z$ is integral.
By Gabber's theorem (\cite[Theorem 1.3 and 1.4]{Il2}), there exist such $X$ as 
in Theorem \ref{thm.cyclemap},
an open subscheme $U$ of $X$, an open subscheme $V$ of $Z$, and a finite \'etale morphism $\pi: U\to V$ of degree prime to $\ell$.
The assertion of the theorem holds for $X$ by Theorem \ref{thm.cyclemap}
and hence for $U$ by the above remark. So it holds for $V$ by a standard norm
argument and hence for $Z$ by the above remark. This completes the proof.
\end{proof}


\section{Application to special values of zeta functions} \label{application2}

\bigskip

\noindent
Let $X$ be a smooth projective variety over a finite field $F$.
We then consider the zeta function
$$
\zeta(X,s)=\underset{x\in \Xd 0}{\prod} \frac{1}{1-N(x)^{-s}}
\quad (s\in \bC)
$$
where $N(x)$ is the cardinality of the residue field $\k(x)$ of $x$.
The infinite product converges absolutely in the region 
$\{s\in \bC\;|\; \Re(s) >\dim(X)\}$ and is known to be continued to
the whole $s$-plane as a meromorphic function. 
Indeed the fumdamental results of Grothendieck and Deligne imply 
\[
\zeta(X,s)=\underset{0\leq i\leq 2d}{\prod} P^i_X(q^{-s})^{(-1)^{i+1}},
\] 
where
$P^i_X(t)\in \bZ[t]$ and for an integer $r$
\[
\zeta(X,r)^*:= \us{s \to r}{\mathrm{lim}} \
 \zeta(X,s) \cdot (1-q^{r-s})^{\rho_r}
\]
is a rational number, where $\rho_r=-\ord_{s=r}\zeta(X,s)$.
The problem is to express these values in terms of arithmetic invariants associated to $X$.
It has been studied in \cite{M2} (where \'etale cohomology is used) and 
in \cite{Li1} and \cite{M3} (where \'etale motivic complexes used)
and in \cite{Li2}, \cite{Ge2} and \cite{Ge3} (where Weil-\'etale cohomology is used).
As an application of Theorem \ref{thm.cyclemap}, we get the following new result
on the problem.

\begin{theo}\label{thm.zeta}
Let $X$ be a smooth projective variety over a finite field $F$.
Let $p=\ch(F)$ and $d=\dim(X)$.
\begin{itemize}
\item[(1)]
For all integers $j$, the torsion part $H^j_M(X,\bZ(d))_{\tors}$ of $H^j_M(X,\bZ(d))$
is finite modulo the $p$-primary torsion subgroup.
Moreover, $H^j_M(X,\bZ(d))_{\tors}$ is finite if $d\leq 4$.
\item[(2)]
We have the equality up to a power of $p$:
\begin{equation}\label{eq.zeta}
\zeta(X,0)^*= 
\underset{0\leq j\leq 2d}{\prod} |H^j_M(X,\bZ(d))_{\tors}|^{(-1)^j}
\end{equation}
The equality holds also for the $p$-part if $d\leq 4$.
\end{itemize}
\end{theo}

\begin{rem}\label{rem.zeta}
\begin{itemize}
\item[(1)]
Let $X=\Spec(\cO_K)$ where $\cO_K$ is the ring of integers in a number field.
The formula \eqref{eq.zeta} should be compared with the formula
$$
 \displaystyle \us{s \to 0}{\mathrm{lim}} \
 \zeta(X,s) \cdot s^{-\rho_0}  =
-\frac{\,|H_M^2(X,\bZ(1))_{\tors}| } {|H_M^1(X,\bZ(1))_{\tors}|}\cdot R_K 
$$
which is obtained by rewriting the class number formula using motivic cohomology. 
Thus \eqref{eq.zeta} may be viewed as a geometric analogue of the class number formula.
Note that the regulator $R_K$ does not appear in \eqref{eq.zeta} since
$H^j_M(X,\bZ(d))$ is (conjecturally) finite for $j\not= 2d$.
\item[(2)]
The case $d=2$ of Theorem \ref{thm.zeta} is equivalent to \cite[Proposition 7.2]{K}.
This follows from the following isomorphism for a smooth surface $X$ over a field:
\[
H^{j+2}_M(X,\bZ(2))\simeq H^j(X_{\Zar},\cK_2) \quad\text{ for  } j\geq 0,
\]
which is deduced from \cite[{ Intro.\ }(iv) and (x)]{B} and \cite{NS}, \cite{To}.
\end{itemize}
\end{rem}

\begin{proof}
Put
\[
\CH^d(X,i;\qzl)=\indlim n \CH^d(X,i;\lnz),
\]
\[
{ H^*_{\et}(X,\qzl(d))=\indlim n H^*_{\et}(X,\lnz(d)). }
\]
By Theorem \ref{thm.cyclemap} we have an isomorphism
\begin{equation}\label{eq1.thm.zeta}
\CH^d(X,i;\qzl)\simeq H^{2d-i}_{\et}(X,\qzl(d)).
\end{equation}
By \eqref{HCHfinite} we have an exact sequence
\begin{equation}\label{eq2.thm.zeta}
0\to \CH^d(X,i)\otimes\qzl  \to \CH^d(X,i; \qzl) \to \CH^d(X,i-1)\{\ell\} \to 0.
\end{equation}
Assuming $i\geq 1$, $H^{2d-i}_{\et}(X,\qzl(d))$ is finite by \cite[Theorem 2]{CTSS}. 
Thus \eqref{eq1.thm.zeta} and \eqref{eq2.thm.zeta} imply $\CH^d(X,i)\otimes\qzl=0$
and we get an isomorphism of finite groups
\begin{equation}\label{eq3.thm.zeta}
\CH^d(X,i-1)\{\ell\} \simeq H^{2d-i}_{\et}(X,\qzl(d)).
\end{equation}
In view of \eqref{MCHC}, this shows the first assertion (1).
For the proof of (2), we use the formula
\[
\zeta(X,0)^*=\frac{[H^0_{\et}(X,\bZ)_{\tors}][H^2_{\et}(X,\bZ)_{\cotor}][H^4_{\et}(X,\bZ)]\cdots}
{[H^1_{\et}(X,\bZ)][H^3_{\et}(X,\bZ)][H^5_{\et}(X,\bZ)]\cdots}
\]
due to Milne \cite[Theorem 0.4]{M2}.
Here $H^0_{\et}(X,\bZ)=\bZ$, $H^1_{\et}(X,\bZ)=0$, and
$H^j_{\et}(X,\bZ)$ for $j\geq 3$ and the cotorsion part 
$H^2_{\et}(X,\bZ)_{\cotor}$ of $H^2_{\et}(X,\bZ)$ are finite.
By the  arithmetic Poincar\'e duality we have
\[
H^{2d-i}_{\et}(X,\qzl(d))\simeq \Hom(H^{i+1}_{\et}(X,\zl),\qzl),
\]
where
$H^{i+1}_{\et}(X,\zl)=\projlim n H^{i+1}_{\et}(X,\lnz)$
which is finite for $i\geq 1$.
Thus (2) follows from the following isomorphisms
\[
H^j_{\et}(X,\zl)\simeq H^j_{\et}(X,\bZ)\{\ell\}\qfor j\geq 3,
\]
\[
H^2_{\et}(X,\zl)\simeq H^2_{\et}(X,\bZ)_{\cotor}\{\ell\},
\]
which can be easily shown by using the exact sequence of sheaves
$$
0\to \bZ \rmapo{\ell^n} \bZ\to \lnz\to 0.
$$
\end{proof}


\section{Appendix: Galois cohomology with compact support} \label{duality}

\bigskip

\noindent
In this section we recall Artin-Verdier duality, following \cite[Sec.\ 3]{K}, and we introduce Galois cohomology with compact support.
In this section all sheaf cohomology groups will be with respect the the \'etale topology. 
For a scheme $Z$ and an integer $n>0$ we denote by $D^b_c(Z,\nz)$ the derived category 
of complexes of sheaves of $\Z/n$-modules over $Z$ with bounded constructible cohomology. 
\medbreak

We fix a finitely generated field $K$ and an integer $n>0$ invertible in $K$. 
Note that we have
\[
D^b_c(K,\nz) := D^b_c(\Spec(K),\nz) = \lim_{\lr \atop V} D^b_c(V,\nz) ,
\]
where $V$ runs over all integral schemes of finite type over $\Z$ with function field isomorphic to $K$. For $C\in D^b_c(K,\nz)$ and $a\in \Z$ we endow the \'etale cohomology group $H^a(K,C)$ with the discrete topology.
\medbreak

Let $V$ be an integral scheme of finite type over $\Z$ with function field $K$ and let $j^K:\Spec (K) \to V$ be the canonical morphism. 
Take $C\in D^b_c(K,\nz)$ and choose a dense open subscheme $U\subset V$ and 
$C_U\in D^b_c(U,\nz)$ restricting to $C\in D^b_c(K,\nz)$. Define
\begin{equation}
H^a(V,j_!^K C ) = \lim_{\longleftarrow \atop U'\subset U} H^a(V,j^{U'}_! (C_U\!\! \mid_{U'}))
\end{equation}
where $U'$ runs through all dense open subschemes of $U$ and $j^{U'}:U' \to V$ is the canonical morphism. This inverse limit is independent of the choices we made and we endow it with the pro-finite topology of the inverse limit.
\medbreak

Similarly we define Galois cohomology with compact support for $C\in D^b_c(K,\nz)$ as follows. { Choose an integral scheme $U $ of finite type over $\Z$  with function field $K$ } and choose
$C_U\in D^b_c( U,\nz )$  restricting to $C$. Now we set
\begin{equation}\label{eq.Hc}
H^a_c(K,C) =  \lim_{\longleftarrow \atop V\subset U} \; H^a_c(V,C_V),\quad
H^a_c(V,C_V)=H^a(\Spec(\Z ),\widehat{Rf^{V}_! (C_U\!\! \mid_{V}) })
\end{equation}
where $V$ runs through all dense open subschemes of $U$ and $f^{V}: V\to \Z$ is the canonical morphism and the hat is as in \cite[Sec.\ 3]{K}. Again this group does not depend on the choices made and { it is} endowed with the pro-finite topology of the 
inverse limit.
\medbreak

Assume that $K$ is a one-dimensional global field.
Let $P_K$ be the set of places of $K$ and let $K_v$ for $v\in P_K$ be the henselization of $K$ at $v$.
The next lemma relates Galois cohomology with compact support to the local-global map.  

\begin{lem}\label{locglo.compact}
Fix an integer $q\geq 0$.
Let $C\in D^{b}_c(K,\nz)$ and assume that the cohomology sheaves $H^i(C)$ vanish for 
$i\leq q$.
Then there is a canonical isomorphism of finite groups
\[
\Coker [ H^{q+2}(K, C )  \to \underset{v\in P_K}{\bigoplus} H^{q+2}(K_{v}, C) ]  
\cong  H_c^{q+3}(K,C).
\]
\end{lem}
\begin{proof}
Let $V$ be an integral scheme, separated and of finite type over $\Z$ with function field isomorphic to $K$. There is a long exact localization sequence
\[
\cdots \to  H^a_c( K, C ) \to  H^a( V,j^K_! C ) \to 
\underset{v\in \Sigma_V}{\bigoplus} H^a( K_v , C ) \to  H^{a+1}_c( K,C ) \to  \cdots
\]
where $\Sigma_V$ is the set of places of $K$ which do not correspond to closed points of $V$. Taking the direct limit over all $V$, the lemma follows from the following two claims which are consequences of cohomological vanishing theorems:

\begin{claim}
\begin{itemize}
\item[(1)] 
The map \;
$\indlim {V} H^{q+2}( V,j^K_! C ) \to H^{q+2}(K,C)$
is surjective.
\item[(2)] 
If $V$ is affine and $n$ is invertible on $V$, we have isomorphisms:
\begin{equation*}
H^{q+3}( V,j^K_! C ) \stackrel{\sim }{\to} H^{q+3}( K, C ) \stackrel{\sim }{\to} 
\underset{v\in P_K^\infty}{\bigoplus}  H^{q+3}( K_v, C ),
\end{equation*}
where $P_K^\infty$ is the set of infinite places of $K$.
\end{itemize}
\end{claim}
\end{proof}

\medskip

Now we state a version of Artin-Verdier duality using Galois cohomology with compact support. 
We refer to \cite[Sec.\ 3]{K} and \cite[Sec.\ 1.1]{JSS} for more details. 
Let $K$ be an arbitrary finitely generated field with $n\in K^\times$. 
For $C\in D^b_c(K,\nz)$ we put
\[
C^\vee= R\, Hom(C,\Z/n)\in  D^b_c(K,\nz).
\]

For $C\in D^b_c(K,\nz)$, choose a connected regular scheme $U$ of finite type over $\Z[1/n]$ and choose $C_U\in D^b_c(U,\nz)$ restricting to $C$  as in \eqref{eq.Hc}.
Artin-Verdier duality says that for any $V\subset U$ open there is a natural isomorphism of finite groups
\begin{equation}\label{app.artinver}
 \Hom( H^a_c(V,C_V(d_K)) ,\Z / n)  \cong H^{2d_K+1-a}(V,C^\vee), 
\end{equation} 
where $d_K=\dim(V)$ (which is the Kronecker dimension of $K$). 

Taking the direct limit over all nonempty open $V\subset U$ in \eqref{app.artinver} we get:
\begin{theo}\label{dual.compact}
There is a natural isomorphism
$$ \Hom_{\rm cont}( H^a_c(K,C(d_K)) ,\Z / n)  \cong H^{2d_K+1-a}(K,C^\vee) ,$$ 
where 
$H^a_c(K,C(d_K)) = \projlim{V\subset U} \; H^a_c(V,C_V)$
is endowed with the pro-finite topology.
\end{theo}

Combining Theorem~\ref{dual.compact} and the Poincar\'e duality for $X$, we get:

\begin{coro}\label{coro.compact0}
Let $X$ be a smooth variety of dimension $d$ over $K$ and 
$\pi: X\to \Spec(K)$ be the canonical moprhism. 
Let $\Lam$ be a locally constant constructible sheaf of $\nz$-modules on $X_{\et}$.
Then there is a natural isomorphism 
\[
H^{2d+2d_K+1-a}_c(K,R\pi_* \Lam^\vee(d+d_K)) \cong \Hom( H^a_c(X,\Lam) ,\nz).
\]
\end{coro}

Combining Lemma~\ref{locglo.compact} and Theorem~\ref{dual.compact}, we get:

\begin{coro}\label{coro.compact}
Assume that $K$ is a one-dimensional global field $K$.
Let $X$ be a smooth affine connected variety of dimension $d$ over $K$ and $C$ be 
a constructible $\nz$-sheaf on $X_{\et}$.
Then there is a natural isomorphism of finite groups
\[
\Hom( H^d_c(X,C^\vee )  ,\nz)\cong 
\Coker [  H^{d+2}(X,C(d+1)) \to \bigoplus_{v\in P_K}  
H^{d+2}(X_{K_v},C(d+1)) ].
\]
\end{coro}

\bigskip

\end{document}